\documentclass[11pt,letterpaper]{amsart}
\usepackage{amsmath,amsthm,amscd,amssymb}

\def\a{\alpha}

\def\d{\delta}

\def\e{\epsilon}

\def\f{\frac}
\def\g{\gamma}
\def\G{\Gamma}

\def\k{\kappa}
\def\lb\{{\left\{}
\def\la{\lambda}
\def\La{\Lambda}
\def\lla{\longleftarrow}
\def\lm{\limits}
\def\lra{\longrightarrow}
\def\dllra{\Longleftrightarrow}
\def\llra{\longleftrightarrow}
\def\n{\nabla}
\def\ngth{\negthickspace}
\def\ola{\overleftarrow}
\def\Om{\Omega}
\def\om{\omega}
\def\op{\oplus}
\def\oper{\operatorname}
\def\oplm{\operatornamewithlimits}
\def\ora{\overrightarrow}
\def\ov{\overline}
\def\ova{\overarrow}
\def\ox{\otimes}
\def\p{\partial}
\def\rb\}{\right\}}
\def\s{\sigma}
\def\sbq{\subseteq}
\def\spq{\supseteq}
\def\sqp{\sqsupset}
\def\supth{{\text{th}}}
\def\T{\Theta}
\def\th{\theta}
\def\tl{\tilde}
\def\thra{\twoheadrightarrow}
\def\un{\underline}
\def\ups{\upsilon}
\def\vp{\varphi}
\def\wh{\widehat}
\def\wt{\widetilde}
\def\x{\times}
\def\z{\zeta}
\def\({\left(}
\def\){\right)}
\def\[{\left[}
\def\]{\right]}
\def\<{\left<}
\def\>{\right>}

\def\tec{Teichm\"uller\ }

\def\SA{\mathcal A}
\def\SB{\mathcal B}
\def\SC{\mathcal C}
\def\SD{\mathcal D}
\def\SE{\mathcal E}
\def\SF{\mathcal F}
\def\SG{\mathcal G}
\def\SH{\mathcal H}
\def\SI{\mathcal I}
\def\SJ{\mathcal J}
\def\SK{\mathcal K}
\def\SL{\mathcal L}
\def\SM{\mathcal M}
\def\SN{\mathcal N}
\def\SO{\mathcal O}
\def\SP{\mathcal P}
\def\SQ{\mathcal Q}
\def\SR{\mathcal R}
\def\SS{\mathcal S}
\def\ST{\mathcal T}
\def\SU{\mathcal U}
\def\SV{\mathcal V}
\def\SW{\mathcal W}
\def\SX{\mathcal X}
\def\SY{\mathcal Y}
\def\SZ{\mathcal Z}

\newcommand{\BA}{\ensuremath{\mathbf A}}
\newcommand{\BB}{\ensuremath{\mathbf B}}
\newcommand{\BC}{\ensuremath{\mathbf C}}
\newcommand{\BD}{\ensuremath{\mathbf D}}
\newcommand{\BE}{\ensuremath{\mathbf E}}
\newcommand{\BF}{\ensuremath{\mathbf F}}
\newcommand{\BG}{\ensuremath{\mathbf G}}
\newcommand{\BH}{\ensuremath{\mathbf H}}
\newcommand{\BI}{\ensuremath{\mathbf I}}
\newcommand{\BJ}{\ensuremath{\mathbf J}}
\newcommand{\BK}{\ensuremath{\mathbf K}}
\newcommand{\BL}{\ensuremath{\mathbf L}}
\newcommand{\BM}{\ensuremath{\mathbf M}}
\newcommand{\BN}{\ensuremath{\mathbf N}}
\newcommand{\BO}{\ensuremath{\mathbf O}}
\newcommand{\BP}{\ensuremath{\mathbf P}}
\newcommand{\BQ}{\ensuremath{\mathbf Q}}
\newcommand{\BR}{\ensuremath{\mathbf R}}
\newcommand{\BS}{\ensuremath{\mathbf S}}
\newcommand{\BT}{\ensuremath{\mathbf T}}
\newcommand{\BU}{\ensuremath{\mathbf U}}
\newcommand{\BV}{\ensuremath{\mathbf V}}
\newcommand{\BW}{\ensuremath{\mathbf W}}
\newcommand{\BX}{\ensuremath{\mathbf X}}
\newcommand{\BY}{\ensuremath{\mathbf Y}}
\newcommand{\BZ}{\ensuremath{\mathbf Z}}

\def\bba{{\mathbb A}}
\def\bbb{{\mathbb B}}
\def\bbc{{\mathbb C}}
\def\bbd{{\mathbb D}}
\def\bbe{{\mathbb E}}
\def\bbf{{\mathbb F}}
\def\bbg{{\mathbb G}}
\def\bbh{{\mathbb H}}
\def\bbi{{\mathbb I}}
\def\bbj{{\mathbb J}}
\def\bbk{{\mathbb K}}
\def\bbl{{\mathbb L}}
\def\bbm{{\mathbb M}}
\def\bbn{{\mathbb N}}
\def\bbo{{\mathbb O}}
\def\bbp{{\mathbb P}}
\def\bbq{{\mathbb Q}}
\def\bbr{{\mathbb R}}
\def\bbs{{\mathbb S}}
\def\bbt{{\mathbb T}}
\def\bbu{{\mathbb U}}
\def\bbv{{\mathbb V}}
\def\bbw{{\mathbb W}}
\def\bbx{{\mathbb X}}
\def\bby{{\mathbb Y}}
\def\bbz{{\mathbb Z}}

\numberwithin{equation}{section}

\theoremstyle{plain}
\newtheorem{thm}{Theorem}[section]
\newtheorem{cor}[thm]{Corollary}

\newtheorem{lem}[thm]{Lemma}
\newtheorem{prop}[thm]{Proposition}

\theoremstyle{definition}

\theoremstyle{remark}
\newtheorem*{remark}{Remark}

\newtheorem*{ack}{Acknowledgments}

\newcommand{\ps}{\f\p{\p s}}
\DeclareMathOperator{\area}{Area}
\DeclareMathOperator{\const}{const}
\DeclareMathOperator{\dista}{dist}
\DeclareMathOperator{\QD}{QD}
\DeclareMathOperator{\id}{id}
\DeclareMathOperator{\im}{Im}
\DeclareMathOperator{\re}{Re}
\DeclareMathOperator{\grad}{grad}
\DeclareMathOperator{\inj}{inj}
\DeclareMathOperator{\maxx}{max}
\DeclareMathOperator{\sll}{sl}
\DeclareMathOperator{\ree}{Re}
\DeclareMathOperator{\vol}{vol}
\newcommand{\qd}{quadratic differential}
\newcommand{\qds}{quadratic differentials}
\newcommand{\WP}{Weil-Petersson }
\newcommand{\Tec}{Teichm\"uller }

\begin{document}

\title{}
\title{The Weil-Petersson Hessian of Length on Teichm\"uller Space}

\today

   \author[M. Wolf]{Michael Wolf}
   \address{Rice University}
   \thanks{Partially supported by NSF grants 
DMS-0139877 and DMS-0505603.}
   \email{mwolf@math.rice.edu}
   \urladdr{http://www.math.rice.edu/\textasciitilde mwolf/}

\begin{abstract} We present a brief but nearly self-contained proof
of a formula for the Weil-Petersson Hessian of
the geodesic length of a closed curve (either simple or not simple)
on a hyperbolic surface. The formula is the sum of the integrals of 
two naturally defined positive functions over the geodesic, proving 
convexity of this functional over Teichmuller space (due to Wolpert
(1987)).  We then estimate this Hessian from below in terms of local
quantities and distance along the geodesic. The formula extends to 
proper arcs on punctured hyperbolic surfaces, and the estimate to 
laminations. Wolpert's result that the Thurston metric 
is a multiple of the Weil-Petersson metric directly follows on taking a 
limit of the formula over an appropriate sequence of curves.
We give further applications to upper bounds of the Hessian,
especially near pinching loci, and recover through a geometric
argument Wolpert's result on the convexity of length to the 
half-power.

\end{abstract}

\maketitle


\section{Introduction} One of the foundations of modern \Tec
theory is Wolpert's \cite{Wol87} theorem that the function on \Tec
space that records the geodesic length of a simple closed curve is
convex with respect to the \WP metric. That paper provided a lower bound
for the Hessian of the length function; our purpose here is to present
a brief derivation of a concise formula for the Hessian in terms of 
natural objects on the surface associated to the curve and the tangent
vectors to \Tec space.

To explain this result and add some context, we fix some terminology
and notation.  Let $S$ be a smooth closed surface of genus $g$, 
and let $\ST(S)$ be the \Tec space of (isotopy classes of)
marked hyperbolic structures on $S$. Let $[\g]$ be a free homotopy
class of closed curves, not necessarily simple, a slight
generalization of the setting in which Wolpert worked. Typically $\g$ will
denote the representative of $[\g]$ that is geodesic with respect
to a given metric $g$, with the context making
it clear whether $\g$ is the immersed curve on the surface or an immersion
from a circle into the surface. 

For each hyperbolic surface $(S,g)$, there is a geodesic
representative
$\g=\g_g$ of $[\g]$.  By the uniformization theorem, 
as each point in $\ST(S)$ is represented by a 
unique hyperbolic metric, say $g$, the length of the geodesic
representative $\g$ of $[\g]$ defines a function 
$\ell = \ell_{\g}= \ell_{\g}([g])$ on $\ST(S)$. We investigate the 
second derivative of that function.

The Hessian of a function is well-defined once there is a background metric.
There are many metrics on $\ST(S)$; among the more basic is the \WP
metric.  Representing the tangent space to \Tec space as the space of 
Beltrami differentials which are harmonic with respect to the
hyperbolic metric $g$ representing a point in $\ST(S)$,
the \WP metric is the $L^2$ metric on
that space with respect to the hyperbolic area form $dA_g$.

The first goal of this paper is to write a formula for the \WP Hessian of
the function.

\subsection{Statement of the Formula.} 
In this subsection, we describe our formula in the simplest case; we 
discuss the situation of the second derivative of the length
function $\ell$ of a closed curve,
with the derivative taken
along a \WP geodesic $\G$. Even in this case, we
require some notation.

Let $\G=\G(t)$ a Weil-Petersson geodesic arc; the class
$[\g]$ is represented by the $\G(t)$-geodesic $\g_t$. The tangent vector
to \Tec space at $\G(0)$ is given by a harmonic Beltrami
differential, say $\mu=\f{\bar\Phi}{g_0}$. 

We can extend $\G(0)$-Fermi coordinates along the curve $\g_0$
to complex coordinates in a neighborhood of (a subarc of) 
$\g_0$.  In terms of those coordinates, the quantity  
$-\f{\im\Phi}{g_0}= \im \mu$ is well-defined.
Let $U^{\Phi}$ denote the 
solution to the ordinary differential equation
\begin{equation} \label{UPhiEquation}
U_{yy} - U = -\f{\im\Phi}{g_0},
\end{equation}
where here the geodesic is represented by a vertical 
line in the Fermi coordinate patch with a parametrization
given by arclength.

This is enough terminology so that we
may state our main result as the 

\begin{thm} \label{theorem:HessianFormula}
Along the Weil-Petersson geodesic arc
$\G(t)$, the second variation $\f{d^2}{dt^2}\ell$ of the length 
$\ell(t)=L(\G(t),[\g])$
is given by 
\begin{align} \label{formula:d2L}
\f{d^2}{dt^2}\ell(t) &= \int_{\g_0}- 2(\Delta - 2)^{-1}\f{|\Phi|^2}{g^2_0} ds 
+ \int_{\g_0}
[U^{\Phi}_y]^2 +[U^{\Phi}]^2 ds\\ \notag
&= \int_{\g_0}-2(\Delta - 2)^{-1}\f{|\Phi|^2}{g^2_0}ds \\ \notag
&\hskip.5cm +
\f{1}{2\sinh(\f{\ell}{2})}\iint_{\g_0 \times \g_0}\im\mu(p)[\cosh(d(p,q) - \f{\ell}{2})]\im\mu(q)ds(p)ds(q).
\end{align}
\end{thm}

\subsection{Applications and Extensions.}
The formula \eqref{formula:d2L} lends itself to
applications. Using well-known techniques for solving and estimating
solutions of the relevant differential equations, we obtain results in
a number of different directions. 

Of course, it is immediately apparent that this Hessian is 
positive definite, even for curves which are not simple. The first
summand is written in terms of $2(\Delta -
2)^{-1}\f{|\Phi|^2}{g^2_0}$,
a ubiquitous term in \Tec theory whose definition involves 
the global geometry of $(S,g)$. In Lemma~\ref{subsolution}, we find a
locally defined lower bound for this quantity. 
Thus, when combined with the second expression for the second term,
we obtain an estimate defined only in terms of local quantities or
distance along the geodesic.

In addition to 
providing a means to estimate the Hessian from below, we also obtain
an improvement on the basic convexity result: from an 
easily derived formula for the gradient of length, we will see
(Corollary ~\ref{corollary:two-thirds convex}) immediately
that $\ell^{\f{2}{3}}$ is convex on all of \Tec space. A recent
theorem \cite{wolpert:behavior08} of Wolpert is that this 
may be improved to $\ell^{\f{1}{2}}$ being convex on all of \Tec
space; we provide a proof for that as well. The proof is geometric in the
sense that it hinges on a comparison of two harmonic diffeomorphisms 
of a cylinder.

It is straightforward to extend our derivations both to
laminations and to proper geodesic arcs which connect 
cusps on a hyperbolic cusped surface. The explicit nature
of formula~\eqref{formula:d2L} also lends itself to 
estimates from above, leading to estimates on the 
\WP connection near the pinching locus.  These were also
recently obtained (and announced some time ago) 
by Wolpert \cite{wolpert:behavior08}.

Our final application is a new proof of Wolpert's \cite{wolpert:thurston}
proof that the Thurston metric is (a multiple of) the \WP
metric. We consider a sequence $\{\g_n\}$ of curves whose
geodesic representatives are becoming equidistributed in 
the unit tangent bundle $T^1(S,g)$.  Thurston observed that
$\lim_{n \to \infty} \text{Hess} \ell_{\g_n}$ would be a 
positive definite quadratic form on $T_{(S,g)}\ST(S)$:
by taking a limit of the right-hand side of 
formula~\eqref{formula:d2L}, we see that this Thurston
metric is a multiple of the \WP metric.  Wolpert's argument
followed a more quasiconformal analytic tradition, while 
this derivation is more Riemannian in perspective.

\subsection{Organization of the paper.} We organize the paper as
follows. In Part 1, we derive Theorem~\ref{theorem:HessianFormula}.
Computations in \Tec theory often require fixing a gauge; here
we find it convenient to vary hyperbolic structures $(S,g_t)$
under the condition that the identity mapping 
$\text{id}:(S,g_0) \to (S,g_t)$
is harmonic.  This requirement gives the prescribed
curvature equation a particularly convenient form, and our
derivation begins with a sketch of a useful computation from
\cite{Wo89}.  The rest of the derivation is self-contained,
occupying sections 2-4.

Part 2 of the paper contains the applications and extensions. In
section 5, we find a lower bound for the integrand in the first
term; this leads to an estimate (Corollary~\ref{HessianEstimate})
for the Hessian in terms of an integral along the curve of local
quantities. Section 6 is devoted to a derivation of our
results in the setting of a curve which connects ends of a
(complete) hyperbolic punctured surface. Here, while the
length of an arc is infinite, the Hessian of a regularized
version of the length is positive and finite; section 7 extends our
work from curves to laminations. In section 8,
we use the formulae to give some geometric estimates: we 
quickly prove that 
not only is the length $\ell$ of curves convex, but so is
$\ell^{\f{2}{3}}$.  We then give a longer geometric argument
that $\ell^{\f12}$ is also convex. We give a general upper bound for the 
Hessian, as well as estimates for the \WP connection near
the Deligne-Mumford compactification divisor. Finally, in section 9,
we take a limit of formula \eqref{formula:d2L}
in Theorem~\ref{theorem:HessianFormula} over curves that are
becoming equidistributed
to recover the result that the Thurston metric is a multiple of the
\WP metric.

\begin{ack}  The author is grateful to David Dumas for conversations, 
especially on lengths of laminations; Yair Minsky for pointing out the 
question on convexity of arcs and for encouragement to find cleaner formulae;
Scott Wolpert for comments, particularly on 
section~\ref{sec:One-half convex}; and
especially Zheng (Zeno) Huang for very careful reading, criticism, the
location of some mistakes, and suggestions for improvements.
\end{ack}

\part{A formula for the Weil-Petersson Hessian of length}

We begin with a brief background discussion of the some the 
theory of \tec space and the \WP metric we will need.
Tangent vectors to \tec space at a point $(S,g)$ are represented 
by 'harmonic Beltrami differentials' of the form
$\mu = \f{\overline{\Phi}}{g}$, where
$g$ is a hyperbolic metric on $S$
and $\Phi$ is a holomorphic quadratic differential on $(S,g)$.  
The \WP inner product
of two such tangent vectors is the $L^2$ inner product 
\begin{equation}
\<\f{\overline{\Phi}}{g},\f{\overline{\Psi}}{g}\> = \re\int_S
\f{{\Phi}}{g}\f{\overline{\Psi}}{g}dArea_{g}.
\end{equation}
Much is now known about this metric: by means of an introduction 
to the subject, the \WP metric is not complete
\cite{Chu75} \cite{Wol75}, it is K\"ahler \cite{Ahl}, negatively curved with
good expressions \cite{Roy83} \cite{Tro86} \cite{Wol86} 
for and estimates \cite{Huang05} \cite{LSY04} of the curvatures,
it is quasi-isometric to the pants complex \cite{brock:WPpants03}, 
the isometry 
group is exactly the (extended) mapping class group \cite{MasurWolf:WP02}.
Of course, the stimulus for this article and an important ingredient 
in some of the results above is that the \WP
metric is geodesically convex \cite{Wol87}.

\section{The second derivative of length in space and time.} We are
interested in computing the second variation of geodesic length of a curve
along a Weil-Petersson geodesic. We imagine the setting as a fixed
differentiable surface $S$ equipped with a family of metrics $g_t$,
and on this surface there is a family of curves $\g_t$. The curves 
$\g_t$ are all freely homotopic and may or may not be simple. The defining
equation is that the curves $\g_t$ are $g_t$-geodesics; we shall shortly 
write that equation out in coordinates.

To begin though, we separate the overall second variation of length
into a term that refers only to the second variation of the metric
$g_t$ and a term that refers only to the second variation of the curve
$\g_t$. This separation is quite standard for a variational
functional. Write the length of $\g_s$ in the metric $g_t$ as
$L(g_t,\g_s)$. Then, in this language, the geodesic equation 
takes the form, for all $t$,
\begin{equation}\label{GeodesicEquation}
\ps\biggm|_{s=s_0}L(g_t,\g_s)\[\ps\g_s\] = 0,
\end{equation}
if $\g_{s_0}$ is a $g_t$-geodesic, and $\ps\g_s$ is an infinitesimal 
variation of curves through $\g_{s_0}$.

The second variation of length of the $g_t$-geodesics $\g_t$ is given by
\begin{equation}
\f{d^2}{dt^2}L(g_t,\g_t) = D^2_{11}L(g_0,\g_0)[\dot g,\dot g] + 
2D^2_{12}L(g_0,\g_0)[\dot g] [\dot\g] +
D^2_{22}L(g_0,\g_0)[\dot\g,\dot\g].  \label{BasicExpanded}
\end{equation}
where $\dot g = \f{d}{dt}g_t$ and $\dot\g = \f{d}{dt}\g_t$. 
Of course, if $\g_t$ is a $g_t$-geodesic, then we write the geodesic
equation \eqref{GeodesicEquation} above in this notation as 
$$
D_2L(g_t,\g_t)[\dot\g] = 0
$$
and so
\begin{align*}
0  &= \f d{dt}D_2L(g_t,\g_t)[\dot\g]\\
&= D_1D_2L(g_0,\g_0)[\dot g] [\dot\g] + D_2D_2L(g_0,\g_0)[\dot\g,\dot\g].
\end{align*}
Thus,
\begin{equation}
D^2_{12}L(g_0,\g_0)[\dot g] [\dot\g] =
-D^2_{22}L(g_0,\g_0)[\dot\g,\dot\g].  \label{CrossTerm}
\end{equation}

Substituting \eqref{CrossTerm} into \eqref{BasicExpanded} yields that
\begin{equation}
\f{d^2}{dt^2}L(g_t,\g_t) = D^2_{11}L(g_0,\g_0)[\dot g,\dot g] - 
D^2_{22}L(g_0,\g_0)[\dot\g] [\dot\g].  \label{ddL}
\end{equation}

Some remarks on this equation \eqref{ddL} are in order. First, note
that the term $D^2_{22}L(g_0,\g_0)[\dot\g] [\dot\g]$ is non-negative, as the
surface $(S,g_0)$ is negatively curved; indeed, this second variation
term is positive unless the vector field $\dot\g$ is tangent to the
curve $\g_0$. Thus our task is to prove that the first term
$D^2_{11}L(g_0,\g_0)[\dot g,\dot g]$ 
is larger than the second term.

In the next sections, we evaluate the terms $D^2_{11}L$ and 
$D^2_{22}L$ via different methods.

\section{Second variation of arclength of $\g_0$ in a family of
metrics} We briefly recall the computational scheme of
\cite{Wo89}. Let $\Phi\in\QD(g_0)$ denote a \qd, holomorphic with
respect to a conformal metric $g_0$. 
Then we may consider a family of metrics on $S$ decomposed by type as
\begin{equation} 
g_t = t\Phi dz^2 + g_0\(\SH(t) + \f{t^2|\Phi|^2}{g^2\SH(t)}\)dzd\bar z + 
t\ov\Phi d\bar z^2. \label{pullback}
\end{equation}
Here $z$ is a conformal coordinate for $(S,g_0)$. It is
straightforward 
to check \cite{SchoenYau:univalence} that the metric $g_t$ is hyperbolic if
\begin{equation} \label{laplacian log h}
\Delta_{g_0}\log\SH(t) = 2\SH(t) -\f{2t^2|\Phi|^2}{g_0^2\SH(t)} - 2.
\end{equation}
(Of course, the pullback of a hyperbolic metric by a diffeomorphism is
hyperbolic, and so we might imagine that if we pullback $g_t$ by a
family of diffeomorphisms $\psi_t$, then the result $\psi_t^*g_t$
would also be hyperbolic. Here we have chosen a gauge by requiring
that the identity map 
$\text{id}:(S,g_0)\to(S,g_t)$ is harmonic.)

We are interested in second variations. Differentiating twice and 
applying the maximum principle to the first derivative (see
\cite{Wo89} for an expanded description) yields
\begin{align}\label{Hdotdot}
\dot\SH  &= \f d{dt}\biggm|_{t=0}\SH(t)\equiv0\notag\\
\ddot\SH  &= \f{d^2}{dt^2}\biggm|_{t=0}\SH(t) = 
-2(\Delta - 2)^{-1}\f{2|\Phi|^2}{g^2_0}.
\end{align}
We observe that $-2(\Delta - 2)^{-1}$ is a positive 
operator and so $\ddot\SH\ge0$. Combining \eqref{pullback}  and 
\eqref{Hdotdot}, we conclude that
\begin{align}\label{MetricExpansion}
g_t = g(t) = g_0dzd\bar z &+ t(\Phi dz^2 + \ov\Phi d\bar z^2) \\ \notag
&+ t^2/2\(\f{2|\Phi|^2}{g^2_0} 
+ -2(\Delta-2)^{-1}\f{2|\Phi|^2}{g^2_0}\)g_0dzd\bar z +O(t^4).
\end{align} 

Now use that
\begin{equation*}
D^2_{11}L(g_0,\g_0)[\dot g,\dot g] = \f{d^2}{dt^2}\biggm|_0
L(g_t,\g_0) 
= \f{d^2}{dt^2}\biggm|_{t=0}\int_{\g_0}\sqrt{g_t}.
\end{equation*}
Substituting \eqref{MetricExpansion} into this last integral
with a choice of coordinate so that $\g_0$ is a line $\{\ree z=\const\}$ 
and differentiating under the integral symbol then yields
\begin{align}\label{metricvariation}
D^2_{11}L(g_0,\g_0)[\dot g,\dot g]  &= \int_{\g_0} 
- \f14(g_0)^{-3/2}(2\ree\Phi)^2 \\ \notag
&\hskip1cm + \f12\sqrt{g_0}\(\f{2|\Phi|^2}{g^2_0} 
- 2(\Delta - 2)^{-1}\f{2|\Phi|^2}{g^2_0}\)\\ \notag
&= \int_{\g_0}\{\f{(\im\Phi)^2}{g^2_0} 
- \[2(\Delta - 2)^{-1}\f{|\Phi|^2}{g^2_0}\]\}\sqrt{g_0},  
\end{align}
since $|\Phi|^2 -(\ree\Phi)^2 = (\im\Phi)^2$.
Both terms are positive, and so we see that this expression is
positive, as we expected (and needed if the 
expression $D^2_{11}L-D^2_{22}L$ is to be positive).

\begin{remark} As an easy model of this method, we quickly reproduce
a formula for the first variation of length.  (We'll have use of this
expression in later sections.) 

We compute the first derivative of length
$\ell$ along $\G(t)$ to be 
\begin{equation}
\f{d}{dt} \ell_{\g}(\G(t)) = D_1L(g_t, \g_t)[\dot{g}] + D_2L(g_t,
\g_t)[\dot{\g}].
\end{equation}

Of course, as $\g_0$ is a geodesic, the second term 
$D_2L(g_t,\g_t)[\dot{\g}] =0$ and from \eqref{pullback} and
\eqref{laplacian log h}, we find
\begin{align} \label{FirstVariation}
\f{d}{dt} \ell_{\g}(\G(t)) &= \f{d}{dt} \int_{\g_0}\sqrt{g_t}\\ \notag
&= \int_{\g_0}\f{\f{d}{dt}g_t}{2g_0} \\ \notag
&=\int_{\g_0} \f{\ree \Phi}{g_0} ds,
\end{align} 
by \eqref{MetricExpansion}, concluding the computation.
\end{remark}

\section{Second variation of $g_0$-arclength of the family $\g_t$ of
$g_t$-geodesics} Our next step is to evaluate the term
$D^2_{22}L(g_0,\g_0)=\f{d^2}{dt^2}L(g_0,\g_t)[\dot\g,\dot\g]$ 
in \eqref{ddL}. 

It is of course standard (see for example \cite{Spivak4})
that if $V$ is the variational field of a
family of curves through a geodesic $\g_0$, 
then
\begin{equation}\label{SecondVariationLength}
\f{d^2}{dt^2}L(g_0,\g_t) =  \int_{\g_0}\biggm|\f{\p V}{\p s}\biggm|^2
 - K|V|^2 ds
\end{equation}  
where $K=K(s)$ denotes the Gaussian curvature of the surface
at the point $\gamma_0(s)$. To
compare this formula \eqref{SecondVariationLength} to 
\eqref{metricvariation}, we will need to find an
expression for $V$ in terms of the \qd\ $\Phi$. That is the main goal
of 
this section.

Of course, the defining equation of $\g_t$ is that it is a geodesic,
 or equivalently that its geodesic curvature vanishes. We write this
 schematically, in a similar way that we write the 
length $L=L(g_t,\g_t)$, as a function $\k=\k(g_t,\g_t)$ of a
 metric and a curve:
\begin{equation}
\k(g_t,\g_t) = 0.
\end{equation}
Differentiating in $t$, we find that
\begin{equation} \label{DGeodesicCurvature}
\f d{dt}\k(g_0,\g_0)[\dot\g] = -\f d{dt}\k(g_0,\g_0)[\dot g].
\end{equation}
As expected, the left-hand side of \eqref{DGeodesicCurvature}
is the classical Jacobi
operator $\(\f{d^2}{ds^2}+K\)$, but the right hand side will involve
the first derivatives of $g_t$, i.e. the metric $g_0$ and the \qd\
$\Phi$. Our next task will be to find an expression for the 
solution $\dot\g$ to \eqref{DGeodesicCurvature}.

\subsection{The inhomogeneous Jacobi equation} We first expand the
right hand side of \eqref{DGeodesicCurvature}. 
We pick conformal (Fermi) coordinates $z=x+iy$ so
that the geodesic $\g_0$ is described by $\{x=\const\}$. 
These coordinates are a bit unusual in that the geodesic may
repeatedly visit the same points on the surface: it's possibly
better to regard the geodesic as embedded in the unit tangent
bundle $T^1M$ with $z=x+iy$ its projection to the surface $S$.

Then, invoking a coordinate expression for $\k$ (see \cite{Oprea},
consistent with definitions in \cite{Spivak3} and \cite{Spivak4}), we have
$$
\f d{dt}\k(g_0,\g_0)[\dot g] 
= -\f d{dt}\lb\{\G^1_{22}(t)\f{\sqrt{\det g(t)}}{g_{22}(t)^{3/2}}\rb\}
$$
where $g(t)=g_{ij}(t)$ is defined by (2.1) as
\begin{align}
g(t) &= \begin{pmatrix}
g_{11}(t)   &g_{12}(t)\\
\\
g_{22}(t)    &g_{22}(t)
\end{pmatrix} \\ \notag
&=\begin{pmatrix}
g_0 + 2t\ree\Phi   &-2t\im\Phi\\
\\
-2t\im\Phi    &g_0 - 2t\ree\Phi
\end{pmatrix} + O(t^2) = \begin{pmatrix}
E   &F\\
F   &G\end{pmatrix}.
\end{align} 
We include the very classical notation for the first fundamental form
at the end as it actually simplifies some of our notation; for
example, we write
\begin{equation} \label{kappa(t)}
\k(t) = -\G^1_{22}(t)\f{\sqrt{EG-F^2}}{G^{3/2}},
\end{equation}
where here of course the variables $E=E(t)$, $F=F(t)$, and $G=G(t)$ all depend on $t$.
Now, in this language, suppressing some of the dependence on $t$, we have
\begin{equation} \label{Christoffel(t)}
\G^1_{22}(t) = \f{2GF_y-GG_x-FG_y}{2(EG-F^2)}.
\end{equation}
Since $\k(0)=0$ and $F(0)\equiv0$, we find that
\begin{equation} \label{dGvanishes}
\f\p{\p x}g_{22}(0)= \f\p{\p x}G=0\quad\text{on }\ \g_0.
\end{equation}
We differentiate \eqref{Christoffel(t)} in $t$ and use 
\eqref{dGvanishes} and $F(0)\equiv0$ to find that
\begin{align*}
\f d{dt}\G^1_{22}(t)  &= \f1{2g^2_0} \{-4g_0(\im\Phi)_y 
+ 2g_0(\ree\Phi)_x + 2(\im\Phi)(g_0)_y\}\\
&=  \f1{2g^2_0}\{-2g_0(\im\Phi)_y + 2(\im\Phi)(g_0)_y\};
\end{align*}
here the last equality
folows from the Cauchy-Riemann equations for the 
real and imaginary parts of the holomorphic \qd\ $\Phi$. We 
conclude that
\begin{equation}\label{dChristoffel}
\f d{dt}\G^1_{22}(t) = -\f\p{\p y}\lb\{\f{\im\Phi}{g_0}\rb\}.
\end{equation}
Combining \eqref{DGeodesicCurvature}, \eqref{kappa(t)}and 
\eqref{dChristoffel} yields the equation we will focus on:
\begin{equation} \label{inhomogeneousJacobi}
\f{\p^2}{\p y^2} V - V = -\f\p{\p y}\lb\{\f{\im\Phi}{g_0}\rb\}.
\end{equation}

\begin{remark}
It is easy to compute that the Beltrami differential tangent
to our deformation is given by $\mu= \f{\bar \Phi}{g_0}$.  In that
language, our equation \eqref{inhomogeneousJacobi} becomes
\begin{equation}
V_{yy} - V = \f\p{\p y}\im \mu.
\end{equation}
\end{remark}

\subsection{The Primitive of the Variation Field}
\label{PrimitiveOfVariation}

The right-hand side of \eqref{DGeodesicCurvature} is the 
derivative of the basic quantity $-\f{\im\Phi}{g_0}$
appearing in \eqref{metricvariation}.  This term provides a link
between the two different terms $D^2_{11}L$ and $D^2_{22}L$ of the 
basic expression \eqref{ddL} for the Hessian of length.
To find the final formula for the Hessian of the 
length function, we consider the primitive of 
$V=\dot\gamma$ along $\gamma$ and use this 
to relate the expressions for $D^2_{11}L$ and
$D^2_{22}L = \int_{\gamma_0} V'^2 + V^2$.

In particular, we begin with 
the equation \eqref{inhomogeneousJacobi}
and then start by defining a particular primitive $U$ of $V$. The
procedure is in two steps, as we need to correctly choose the 
constant for the primitive. So first we set

\begin{equation}\label{defineu}
u(y) = \int_a^y V(s) ds  
\end{equation}
Note that we need to check the well-definedness of $u$ 
on $\gamma$, as it is a closed loop; on the other hand it is
enough to check that the period $u(2\pi) - u(0) = \int_{\gamma} V$
vanishes (here using the obvious notation for a pair of
endpoints for the loop). 

For convenience in the sequel, set 
\begin{equation}\label{FDefined}
\mathcal{F} = \f{\im\Phi}{g_0}.
\end{equation}
so that equation \eqref{inhomogeneousJacobi} becomes
\begin{equation}\label{JacobiWithF}
V_{yy} -V= - \mathcal{F}_y 
\end{equation}
Then, for well-definedness of $u$, we observe that
(letting subscripts indicate differentiation in the variable)

\begin{align}
 u(2\pi) - u(0) &= \int_{\gamma} V\notag\\
&=\int_{\gamma} V_{yy} + \mathcal{F}_y\notag\\
&=\int_{\gamma} (V_y +\mathcal{F})_y dy\notag\\
&=0.
\end{align}
Thus $u$ (and $u+c$, for any constant $c$) is well-defined
along $\gamma$.

Next begin again with the equation
\begin{equation}
V_{yy} - V = -\mathcal{F}_y,
\end {equation}
and then note that
\begin{align}
(u_{yy} -u +\mathcal{F})_y &= V_{yy} -V +\mathcal{F}_y\notag\\
&=0.
\end{align}
Thus we have that $u_{yy} -u +\mathcal{F} =c_0$, where $c_0$ is a constant.
In particular if we set
\begin{equation}
U=u+c_0,
\end{equation}
to be another primitive of $V$, then 
\begin{equation}\label{primitive}
U_{yy} -U = -\mathcal{F}.
\end{equation}

The point of all of this is that the positive part of 
the second variation of  length integral \eqref{SecondVariationLength}
will turn out to be the energy of $U$, while the negative
part will once again be the $L^2$ norm of $\mathcal{F}$ along $\ell_0$ (cancelled 
out by a term in the metric variation contribution $D^2_{11}L$). 

We compute the contribution $-D^2_{22}L$
from the second variation of length along
the surface through
\begin{align}
-D^2_{22}L &=-\int_{\g_0} V_y^2 +V^2 \\
&= \int_{\g_0} V_{yy}V - V^2 \qquad\text{by
parts}\notag \\
&= \int_{\g_0} (V_{yy} - V)V \notag \\
&= \int_{\g_0} (-\mathcal{F}_y)V \qquad\text{by \eqref{inhomogeneousJacobi}}\notag \\
&= \int_{\g_0} \mathcal{F}V_y \qquad\text{by parts} \notag \\
&= \int_{\g_0} \mathcal{F}U_{yy} \qquad\text{from the definition of
} U \text{ as a primitive of } V\notag \\
&= \int_{\g_0} \mathcal{F}(U-\mathcal{F}) \qquad\text{from \eqref{primitive}} \notag \\
&= -\int_{\g_0}\mathcal{F}^2 +   \int_{\g_0} U\mathcal{F}   \notag \\
&= -\int_{\g_0}\mathcal{F}^2 + \int_{\g_0} U\{-(U_{yy} -U)\} \qquad\text{from
\eqref{primitive}} \notag \\
&= -\int_{\g_0}\mathcal{F}^2 + \int_{\g_0} U_y^2 +U^2 \qquad\text{by parts}.
\end{align} \label{cleanSecondVariation}

Combining this last equation with \eqref{metricvariation} and
\eqref{SecondVariationLength},
we find that 
\begin{align} 
\f{d^2}{dt^2}L(g_t,\g_t) &= D^2_{11}L(g_0,\g_0)[\dot g,\dot g] - 
D^2_{22}L(g_0,\g_0)[\dot\g] [\dot\g] \notag \\
&= \int_{\g_0} \mathcal{F}^2 
- \[2(\Delta - 2)^{-1}\f{|\Phi|^2}{g^2_0}\] -
\int_{\g_0}V_y^2 + V^2 \notag \\
&=\int_{\g_0} \mathcal{F}^2- \[2(\Delta - 2)^{-1}\f{|\Phi|^2}{g^2_0}\]
\\ \notag
&\hskip1cm- \int_{\g_0}\mathcal{F}^2 +
\int_{\g_0} U_y^2 +U^2 \qquad \text{from \eqref{cleanSecondVariation}}
\notag\\
&= \int_{\g_0}- 2(\Delta - 2)^{-1}\f{|\Phi|^2}{g^2_0} + \int_{\g_0}
U_y^2 +U^2. 
\end{align}  \label{HessianwithU}

In summary, the Weil-Petersson Hessian of length can be expessed as the 
sum of two integrals along the curve, each of which has a positive
function
as an integrand. The first integrand
is the restriction to the curve of a solution of a differential equation on the surface, and the second is the 
energy density of a solution of a differential equation along the curve.

We record this formula as a theorem, extending
Theorem~\ref{theorem:HessianFormula} from the introduction. 
To set the notation,
let $[\g]$ be the free homotopy class of a closed curve (simple or
not) on the surface,
and $\G(t)$ a Weil-Petersson geodesic arc; the class
$[\g]$ is represented by the $\G(t)$-geodesic $\g_t$. The tangent vector
to \Tec space at $\G(0)$ is given by a harmonic Beltrami
differential, say $\f{\bar\Phi}{g_0}$. Let $\f{\bar\Psi}{g_0}$
denote a second harmonic Beltrami differential on $\G(0)$.

Let $U^{\Phi}$ and $U^{\Psi}$ denote the respective 
solutions to the ordinary differential equations (see \eqref{UPhiEquation})
\begin{equation}\label{UPhiEquationb}
U_{yy} - U = -\f{\im\Phi}{g_0}
\end{equation} 

and 
\begin{equation} \label{UPsiEquation}
U_{yy} - U = -\f{\im\Psi}{g_0}.
\end{equation}
This is enough terminology so that we
may summarize our discussion as

\begin{thm} \label{theorem:PolarHessianFormula}
Along the Weil-Petersson geodesic arc
$\G(t)$, the second variation $\f{d^2}{dt^2}\ell$ of the length 
$\ell(t)=L(\G(t),[\g])$
is given by 
\begin{equation}\label{d2lintheorem}
\f{d^2}{dt^2}\ell(t)= \int_{\g_0}- 2(\Delta - 2)^{-1}\f{|\Phi|^2}{g^2_0} ds 
+ \int_{\g_0}
[U^{\Phi}_y]^2 +[U^{\Phi}]^2 ds.
\end{equation}

More generally, the Weil-Petersson Hessian $\text{Hess} L[\f{\bar\Phi}{g_0},
  \f{\bar\Psi}{g_0}]$
is given by 

\begin{equation} \label{formula:HessL}
\text{Hess} L[\f{\bar\Phi}{g_0},
  \f{\bar\Psi}{g_0}]=
\int_{\g_0}- 2(\Delta - 2)^{-1}\f{\re\Phi\bar\Psi}{g^2_0} ds + \int_{\g_0}
U^{\Phi}_yU^{\Psi}_y +U^{\Phi}U^{\Psi} ds.
\end{equation}
\end{thm}

\begin{proof}
The solutions $U^{\Phi}$ and $U^{\Psi}$ to 
\eqref{UPhiEquationb} and \eqref{UPsiEquation}
are unique, and the equations are linear
in the unknown and the parameters 
$\Phi$ and $\Psi$. Thus the unique solution
$U^{\Phi + \Psi}$ to
\begin{equation} 
U_{yy} - U = -\f{\im(\Phi + \Psi)}{g_0}
\end{equation}
satisfies
\begin{equation}\label{UPhiPsi}
U^{\Phi + \Psi}=U^{\Phi} + U^{\Psi}.
\end{equation} 
Then a straightforward polarization of \eqref{formula:d2L},
together with our understanding \eqref{UPhiPsi},
yields \eqref{formula:HessL}.
\end{proof}

\begin{remark} In terms of our previous notation for the 
Beltrami differential $\mu = \f{\bar\Phi}{g_0}$, the equation
\eqref{UPhiEquationb} takes the form
\begin{equation}
U_{yy} - U = \im\mu.
\end{equation}
\end{remark}

\subsection{A geometric kernel representation} \label{subsection:kernel}
The second term of equation \eqref{formula:d2L} is expressed as the
energy of the solution of a differential equation.  We wish to provide
a more geometric interpretation; not only do we hope that 
this version is more appealing on its own, but it will be
important in section~\ref{Thurston} where we treat the 
Thurston metric via the Hessian of the length function.

To begin, note that the second 
term of \eqref{formula:d2L} (=\eqref{d2lintheorem})
may be written
\begin{align*}
\int_{\g_0} U_y^2 +U^2ds &= -\int_{\g_0} U\{(U_{yy} -U)ds\\
&=\int_{\g_0}U(s)\SF(s)ds
\end{align*}
in the notation where $\SF(s)= \im \f{\Phi}{g_0}= -\im\mu$, since
\begin{equation*}
U_{yy} -U = -\mathcal{F}
\end{equation*}
by \eqref{primitive}.  It is well-known that we can represent
the solution $U(s)$ to \eqref{primitive} by
\begin{equation*}
U(s)= -\int \SF(t)K(s,t)dt
\end{equation*}
for kernels $K(s,t)$ which satisfy
\begin{equation} \label{KernelEquation}
\f{d^2}{dt^2}K(s,t) - K(s,t) = \delta_s(t).
\end{equation}
It is easy to guess the solution to \eqref{KernelEquation}
using that the solution to $\f{d^2}{dt^2}K(s,t) - K(s,t) =0$ are
linear combinations of $\sinh(t)$ and $\cosh(t)$.  (Indeed,
if we represent $\g_0$ as the interval $[-L/2,L/2]$ with 
endpoints identified, set $s_0=\pm L/2$ to be an endpoint, and look
to solve \eqref{KernelEquation} on that interval, then
it is evident that setting
$K(s,t) = -\cosh(t)$ is correct up to an easily computed multiplicative
constant.) In general, for $\gamma$ parametrized by an interval 
of length $L$ (so that we may choose $|t-s|<L/2$, we have that
\begin{equation*} 
K(s,t) =\begin{cases} -\f{1}{2} \f{\cosh(s-t-L/2)}{\sinh(\f{L}{2})}, & t<s\\
        -\f{1}{2} \f{\cosh(t-s-L/2)}{\sinh(\f{L}{2})}, & t>s
        \end{cases}
\end{equation*}
solves \eqref{KernelEquation} (where we require that $|t-s|<L/2$).

Of course, the variables $s$ and $t$ parametrize the curve $\g_0$ 
with respect to arclength, and so, for $|s-t|<L/2$, we have 
$|s-t|=d(\g_0(s),\g_0(t))$. Thus $K(s,t)$ admits the description in
terms of $p=\g_0(s),q=\g_0(t)$ as
\begin{equation}
K(p,q)=  -\f{1}{2} \f{\cosh(d(p,q)-L/2)}{\sinh(L/2)}.
\end{equation}

This leads to the representation
\begin{align}
\f{d^2}{dt^2}\ell(t) &= \int_{\g_0}-2(\Delta -2)^{-1}\f{|\Phi|^2}{g^2_0} 
+ \int_{\g_0}U_y^2(s) +U^2(s)ds \notag \\
&= \int_{\g_0}-2(\Delta - 2)^{-1}\f{|\Phi|^2}{g^2_0} -
\int_{\g_0}U(s)\im\mu(s)ds \notag \\
&= \int_{\g_0}-2(\Delta - 2)^{-1}\f{|\Phi|^2}{g^2_0} -
\int_{\g_0}\im\mu(s) \int_{\g_0}K(s,t)\im\mu(t)dtds \notag \\
&= \int_{\g_0}-2(\Delta - 2)^{-1}\f{|\Phi|^2}{g^2_0} - \iint_{\g_0 \times
  \g_0}\im\mu(s)K(s,t)\im\mu(t)dtds \notag \\
&= \int_{\g_0}-2(\Delta - 2)^{-1}\f{|\Phi|^2}{g^2_0} - \iint_{\g_0 \times
  \g_0}\im\mu(p)K(p,q)\im\mu(q)ds(p)ds(q) \notag \\
&= \int_{\g_0}-2(\Delta - 2)^{-1}\f{|\Phi|^2}{g^2_0}\\ 
&\hskip.3cm+\f{1}{2\sinh(\f{L}{2})}\iint_{\g_0 \times \g_0}\im\mu(p)[\cosh(d(p,q) - \f{L}{2})]\im\mu(q)ds(p)ds(q).\label{formula:KernelFormula}
\end{align} 
where $ds(p)$ and $ds(q)$ refer to arclength measure.

\part{Extensions and Applications of the formula for the Hessian.}

\section{A lower bound expressed in terms of pointwise quantities.}\label{LowerBound}

We claim

\begin{lem}\label{subsolution} 
Let $v=1/3\f{|\Phi|^2}{g^2_0}$. Then
  $v$ is a subsolution of 
$(\Delta-2)f =-\f{2|\Phi|^2}{g^2_0}$ and in 
particular $0\le v\le-2\Delta-2)^{-1}(\f{ |\Phi|^2}{g^2_0})$.
\end{lem}

We begin by noting that the curvature of a metric expressed as $G|dz|^2$ is given by
$$
K(G|dz|^2) = -\f12\f1G\Delta_0\log G
$$
where $\Delta_0=\p^2_x+\p^2_y$.

Then using that $K(g_0|dz|^2)\equiv-1$ and that $|\Phi_0||dz|^2$ is a
flat metric with 
concentrated (Dirac function type) curvature singularities at the
zeroes $\Phi^{-1}(0)$ 
of $\Phi$, we see that
\begin{align*}
\Delta_0\log\f{|\Phi|^2}{g^2_0}  &= \Delta_0\log|\Phi|^2 - \Delta_0\log g^2_0\\
&= -4|\Phi| K(|\Phi| |dz|^2) + 4g_0 K(g_0)\\
&= 4|\Phi|\sum_{p\in\Phi^{-1}(0)}\pi\delta_p\deg_p\Phi - 4g_0,
\end{align*}
where $\delta_p$ indicates a delta function at $p$.
On the other hand, using that $\Delta_0\log F=\f{\Delta_0 F}F-\f{|\n_0
  F|^2}{F^2}$, 
we see we may write
$$
\Delta_0\log\f{|\Phi|^2}{g^2_0} = \f{\Delta_0\f{|\Phi|^2}{g^2_0}}{\f{|\Phi|^2}{g^2_0}} - \f{|\n_0\(\f{|\Phi|^2}{g^2_0}\)|^2}{\(\f{|\Phi|^2}{g^2_0}\)^2}.
$$
Putting the last two of these equations together yields
$$
\f1{g_0}\Delta_0\f{|\Phi|^2}{g^2_0} = \f{|\Phi|^2}{g^2_0}\(4\f{|\Phi|}{g_0}\sum_{p\in\Phi^{-1}(0)}\pi\delta_p\deg_p\Phi\) - 4 \f{|\Phi|^2}{g^2_0} + \f{|\n_0\(\f{|\Phi|^2}{g^2_0}\)|^2}{\f{|\Phi|^2}{g^2_0}g_0}
$$
In particular writing $\Delta=\f1{g_0}\Delta_0$ for the $g_0$-Laplace
Beltrami operator 
on $S$, and noting the vanishing of the first term on the right hand side, we conclude that
$$
\Delta\f{|\Phi|^2}{g^2_0}\ge-4\f{|\Phi|^2}{g^2_0}.
$$
We are of course interested in the operator $\Delta-2$, so we note the
obvious 
implication that
$$
(\Delta - 2)\f{|\Phi|^2}{g^2_0}\ge-6\f{|\Phi|^2}{g^2_0}
$$
so that $v=1/3\f{|\Phi|^2}{g^2_0}$ is a 
subsolution for the equation $(\Delta - 2)f=-2\f{|\Phi|^2}{g^2_0}$.

It is obvious that $v=\f13\f{|\Phi|^2}{g^2_0}\ge0$, and if $f$
satisfies $(\Delta-2)f = -\f{2|\Phi|^2}{g^2_0}$, then
$\Delta(f-v)\le2(f-v)$ and so the minimum principle guarantees that at
a minimum of $(f-v)$, we have $f-v\ge0$; hence $f-v\ge0$ everywhere, 
concluding the proof of the lemma. \qed

Combining Lemma~\ref{subsolution} with
Theorem~\ref{theorem:HessianFormula} 
we obtain

\begin{cor}\label{HessianEstimate} 
Let $\Phi\in\QD(g_0)$ be a holomorphic \qd\ in $(\Sigma,g_0)$ and let
$\G(t)$ denote a Weil-Petersson geodesic arc with initial tangent
vector given by the harmonic Beltrami differential $\ov\Phi
g^{-1}_0$. Let $\ell(t)$ denote the geodesic length of a
representative $\g_t$ of a curve class $[\g]$ on $S$. 
Then, for $\g_0$ the geodesic represented of $[\g]$ on $\G(0)$, we have
$$
\f{d^2}{dt^2}\biggm|_{t=0}\ell(t)\ge\f{1}{3}\int_{\g_0}\f{|\Phi|^2}{g^2_0}ds.\qed
$$
\end{cor}

We will apply this estimate in the section~\ref{sec:Two-thirds convex}.

\section{The Second Variation of the length of an Arc}

In this section, we adapt our derivation to the case where $S$ is a
surface of finite genus with a finite number of punctures, and we are
interested in the variation of length of an arc ${\a}$ that runs between
two of the punctures (or a puncture itself). Naturally, the length of
such an arc is infinite, so we will be discussing the variation of some
regularization of its length; nevertheless, all of the basic
considerations will extend to this case with only minor 
modifications.

\subsection{Notation and preliminaries} Let ${\a}_t$ be the geodesic on
$(S,g_t)$ that connects punctures $p$ and $q$ in a fixed homotopy
class (rel $p$ and $q$). Consider a sequence of points $p_n$,
$q_n\in {\a}_0$ with $p_n\to p$ and $q_n\to q$ and let ${\a}_{t,n}$ denote
the finite length $g_t$-geodesic arc connecting $p_n$ to $q_n$ which
is homotopic (rel $p_n$, $q_n$) to ${\a}_{0,n}\subset {\a}_0$. 
Our plan is to derive a formula for 
$$
\f{d^2}{dt^2} L(g_t,{\a}_{t,n}),
$$
show that the limit exists and is independent of the choice of the 
sequence $\{p_n,q_n\}$.

We learned while preparing this manuscript that Wolpert 
\cite{wolpert:WPextension07} recently 
treated the analogous case of finite length
arcs between horocycles.

It is easy to check that the formal preliminaries remain the same as in
the 
derivation of \eqref{ddL}, and so we conclude
\begin{equation}\label{ArcDifference}
\f{d^2}{dt} L(g_t,{\a}_{t,n}) = D^2_{11} L(g_0,{\a}_{0,n})[\dot g,\dot g] 
- D^2_{22} L(g_0,{\a}_{0,n})[\dot{\a}_{0,n},\dot{\a}_{0,n}].
\end{equation}

\subsection{The Second Variation of Arclength of ${\a}_{0,n}$ in $g_t$}
As in the case of a closed curve, the first term in
\eqref{ArcDifference} is relatively straightforward to compute; the
only new issue to consider is the dependence of the term on the choice
of endpoints $p_n$, $q_n$ of ${\a}_{t,n}$. Indeed, 
exactly as in the derivation of \eqref{metricvariation}                    , 
we formally compute
\begin{equation}\label{D11n}
D^2_{11} L(g_0,{\a}_{0,n}) = \int_{\a_{0,n}}\lb\{\f{|\im\Phi|^2}{g^2_0} 
-2(\Delta - 2)^{-1}\f{|\Phi|^2}{g^2_0})\rb\}\sqrt{g_0}
\end{equation}
where the principal issue is to determine the meaning of 
$(\Delta-2)^{-1}\f{|\Phi|^2}{g^2_0}$.

\subsection{Variations of metrics of finite area}

The basic point here 
in understanding $-2(\Delta-2)^{-1}\f{|\Phi|^2}{g^2_0}$ is
to construe it as $\ddot\SH$ for the family of pullback
metrics $g_t$ in \eqref{pullback}. As these maps 
$\text{id}: (S, {g_0}) \to (S, g_t)$ are harmonic, we can apply some results 
from the theory of harmonic maps between cusped hyperbolic surfaces.


In this direction, results in \cite{Wo91a} (Theorem~5.1) and of 
Lohkamp (see the remark after Theorem 4 in
\cite{loh91}, especially with Lemma~12 informed by Proposition~3.13 
in \cite{Wo91a}) proved that $\SH(t)\in C^{k,\a}(S,g_0)$ was analytic in
$t$ on the compactified surface $\bar{M}$.
In particular $\ddot\SH$ is bounded.

Indeed, we can easily show from this that 
$\ddot\SH=O(\f1{(\log\f1r)^\alpha}) = O(y^{-\a})$
for some $\alpha\in(0,1)$ as $r\to0$. The basic elements of this argument is
that $\f1{(\log\f1r)^\alpha}=y^{-\a}$ for some $\alpha\in(0,1)$ is a 
supersolution of the equation
$(\Delta-2)\ddot\SH=-2\|\Phi\|^2$ on the cusp, as well as the point that the kernel
of $(\Delta-2)$ on a half-infinite cylinder $\SC=\{\im z>1,|\ree
z|<1/2\}$ (with the standard identifications) is spanned by 
the pair of functions $k_1(z)=y^2$
and $k_2(z)=y^{-1}$. With that background, consider on the finite
cylinder $\{1<\im z<y_n\}$, a function $H_j(z)$ of the form
$H_j(z)=C_0y^{-\alpha}+C_1y^{-1}+\e_jy^2$.

Then for appropriate choices of
$\e_j\to0$, we find that $H_j(z)$  majorizes $\ddot\SH$; letting
$j\to\infty$ and $\e_j\to0$ while $C_0$ and $C_1$ stay bounded (as the
boundary values $\ddot\SH(z)$ for $\{\im z=1\}$ are fixed
independently of $j$) allows us to conclude that $\ddot\SH(z)$ decays
like $C_0y^{-\alpha}+C_1y^{-1}$. Thus
$\ddot\SH(z)=O(y^{-\alpha})=O(\f1{(\log\f1r)^{\alpha}})$.

Looking ahead to the final form of the second variation of length, 
it is worth recording the 

\begin{prop} \label{First Term Converges} Let $p_n\to p$ and  $q_n\to
  q$, and let ${\a}_{0,n}$ be the geodesic arc connecting $p_n$ to $q_n$
  as 
in the introduction to the section. Then
$$
\int_{{\a}_{0,n}}-2(\Delta - 2)^{-1}\f{|\Phi|^2}{g^2_0} ds 
= \int_{{\a}_{0,n}}\f12\ddot\SH ds
$$
converges as $n\to\infty$.
\end{prop}

\begin{proof} In the upper half plane coordinates, we have that for
  each 
end of an ${\a}_{0,n}$,
\begin{equation}
\int_{{\a}_n}\ddot\SH ds = \int^{y_n}_a\ddot\SH\f{dy}y = 
\int^{y_n}_aO(\f1{y^{\alpha}})\f{dy}y = O(1)
\quad\text{as }\ n\to\infty.
\end{equation}
The proposition then follows from the integrals being positive.
 \end{proof}

 \subsection{The Jacobi Field for an Arc}
 
The next term we must address is the second term in
\eqref{ArcDifference}. The variational vector field, defined
geometrically, also satisfies 
\eqref{inhomogeneousJacobi}. Here, of
course the variational field $V_n$ depends on $n$, as it is defined in
terms of $\alpha_{t,n}$; our notation is meant to reflect that. The 
form of the second
variation of arclength is also unchanged at
$\int_{{\a}_{0,n}}V^2_{n,y}+V^2_n$, as the boundary term vanishes once
we require the family of curves ${\a}_{t,n}$ to have $p_n$ and $q_n$ as 
endpoints, independently of $t$.

The primitive
$U_n=\int^y_aV_n(s)ds$ exists as before --- in fact, in the setting of
an open arc, there is not a well-definedness issue to check, though
we still have to adjust by a constant.  In particular,
note that for $\hat{U}_n = \int_{p_n}^y V_{n}(s) ds$
(so that $\hat{U}_n = \int_{p_n}^y V_{n}''(s) - (\im \mu)'(s)ds$), we have
$\hat{U}_n'' - \hat{U}_n = \im \mu + (V_n'(p_n) - \im \mu(p_n))$.
So set $U_n = \int_{p_n}^y V_{n}(s) ds + (V_n'(p_n) - \im \mu(p_n))$.
Then $U_n$ solves the (boundary value) problem
\begin{align} \label{Unbdy}
U_n'' - U_n &= \im \mu \\
U_n(p_n) &= V_n'(p_n) - \im \mu(p_n) \notag \\
U_n(q_n) &= V_n'(q_n) - \im \mu(q_n). \notag
\end{align}
The last boundary condition follows after applying the 
fundamental theorem of calculus to 
\begin{align*}
U_n(q_n) &= \int^{q_n}_{p_n} V_n(s)ds + V_n'(p_n) - \im \mu(p_n)\\
&= \int^{q_n}_{p_n} V_n''(s) - (\im \mu)'(s) ds + V_n'(p_n) - \im
\mu(p_n)
\end{align*}
We will soon show that these boundary values, though non-zero, tend to 
zero and have no effect on the limiting relation.

At this stage, it is useful to write a kernel representation for
$V_n$. Let $L_n = d(p_n,q_n)$, and parametrize the arc $\a_{0,n}$ by
$[-\f{L_n}{2}, \f{L_n}{2}]$. It will not ultimately affect the  results
if in these coordinates the midpoints of $[p_n, q_n]$ (represented in 
our coordinates by the origin) do not remain in a compact set.

The kernel for the operator $\f{d^2}{dy^2}-1$ on the segment 
$[-\f{L_n}{2}, \f{L_n}{2}]$ is given by
\begin{equation} \label{kernelKn}
K_n(y,s)=\begin{cases}
-\f{\sinh(\f{L_n}{2}+s)\sinh(\f{L_n}{2}-y)}{\sinh(L_n)},
&-\f{L_n}{2}\le s \le y \le \f{L_n}{2}\\
-\f{\sinh(\f{L_n}{2}-s)\sinh(\f{L_n}{2}+y)}{\sinh(L_n)},
&-\f{L_n}{2}\le y \le s \le \f{L_n}{2}.
\end{cases}
\end{equation}
This gives the representations
\begin{equation} \label{VnRep}
V_n(y) = \int_{p_n}^{q_n} K_n(y,s) (\im \mu)'(s) ds
\end{equation}
and
\begin{equation} \label{UnRep}
U_n(y) = \int_{p_n}^{q_n} K_n(y,s) (\im \mu)(s) ds + a_n\cosh(y) + b_n\sinh(y),
\end{equation}
where $a_n$ and $b_n$ are chosen to satisfy the boundary conditions in
\eqref{Unbdy}.

We will estimate the asymptotics of these boundary terms in the 
next subsection, but we display the preliminaries for that analysis
here, in the present context of integral formulas for the relevant
geometric objects.

One of the terms in the boundary condition is given by $V_n'$ at
the corresponding boundary point. For example, $V_n'(p_n)$
is given by 
\begin{equation} \label{Vnprime}
V_n'(p_n) = \int_{p_n}^{q_n} \f{d}{dy}\bigg|_{y=p_n^+}K_n(y,s) (\im
\mu)'(s) ds.
\end{equation}
We then compute that
\begin{equation} \label{Knprime}
\f{d}{dy} K_n(y,s)= \begin{cases} \f{\sinh(\f{L_n}{2}+s)\cosh(\f{L_n}{2}-y)}{\sinh(L_n)},
&-\f{L_n}{2}\le s \le y \le \f{L_n}{2}\\
-\f{\sinh(\f{L_n}{2}-s)\cosh(\f{L_n}{2}+y)}{\sinh(L_n)},
&-\f{L_n}{2}\le y \le s \le \f{L_n}{2}
\end{cases}.
\end{equation}
Combining \eqref{Vnprime} and \eqref{Knprime} yields the
representation
\begin{equation} \label{Vnprimebdy}
V_n'(p_n) = -\int_{-\f{L_n}{2}}^{\f{L_n}{2}} \f{\sinh(\f{L_n}{2}-s)}{\sinh(L_n)}(\im
\mu)'(\a_n(s)) ds.
\end{equation}

The final matter is the analogue of the 
derivation of \eqref{HessianwithU} from the 
variation of
arclength. This involves three integrations by parts, and so we need
to consider the boundary terms from each integration; however, all of
the terms have either $V_n$ or $U_{n,y}=V_n$ as a factor, and so all
vanish at the endpoints $p_n$ 
and $q_n$.
 
We conclude that
\begin{equation}\label{D22n}
-D^2_{22}L(g_0,{\a}_{0,n}) (\dot{\a}_{0,n},\dot{\a}_{0,n}) 
= -\int_{{\a}_{0,n}}\mathcal{F}^2 + \int_{{\a}_{0,n}}U^2_{n,s} +U^2_n 
\end{equation}
in the notation of \eqref{FDefined}.

We combine formulae \eqref{D11n} 
and \eqref{D22n} for the derivatives $D_{11}L(g_0,{\a}_{0,n})$
 and $D_{22}L(g_0,{\a}_{0,n})$ to find
\begin{equation}\label{lengthn}
\f{d^2}{dt^2} L(g_t,{\a}_{t,n}) = \int_{{\a}_{0,n}} - 
2(\Delta - 2)^{-1}\f{|\Phi|^2}{g^2_0} + \int_{{\a}_{0,n}} U^2_{n,s} +U^2_nds.
\end{equation}

\subsection{Passage to the limit}
Naturally, we are interested in taking the limit 
of \eqref{lengthn} as $n\to\infty$. That
the first term converges in the content of Proposition~\ref{First Term 
Converges}.

For the second term, we need to estimate the asymptotics of
$U_{n,y}=V_n$ 
and $U_n$. We claim

\begin{prop} \label{Vn limit} The fields $V_n$ converge to a field $V$
  defined on the entire arc ${\a}_0$. The primitives $U_n$ may be chosen
  so that, 
not only does $U_n$ converge to a field $U$ with $U_s=V$, but also
\begin{equation}
\int_{{\a}_{0,n}}U^2_{n,s} + U^2_n\lra\int_{{\a}_0} U^2_s + U^2 < \infty.
\end{equation}
Here $U_s$ and $U_{n,s}$ indicate derivatives of $U$ and $U_n$ with
respect to the arclength parameter $s$.
\end{prop}

\begin{proof} 

It is useful to begin with an observation.

\begin{lem} \label{lemma:mudecay} 
Near an end of $(S,g_0)$ with coordinates from $\{|z|<1\}$, we have
$|\mu| = O(r(\log\f{1}{r})^2)$ and $|\f{d}{ds}\mu|= O(r(\log\f{1}{r})^3)$.
\end{lem}

\begin{proof} In the coordinate disk, the holomorphic quadratic
  differential $\Phi = \f{c}{z} + h.o.t$, while 
$\f{1}{g_0} = r^2(\log\f{1}{r})^2$ and
  $\f{d}{ds}=r\log\f{1}{r}\f{\partial}{\partial r}$ along radial 
geodesics.  The estimates are then immediate.
\end{proof}

Recall the terms $a_n$ and $b_n$ in \eqref{UnRep} that adjust the
solution $U_n$ for the inhomogeneous boundary conditions.
Our next goal is to prove that these adjustments are asymptotically
inconsequential.

\begin{lem} \label{anbndecay}
$a_n,b_n = o(e^{-c\f{L_n}{2}})$.
\end{lem}
\begin{proof}  The goal is to show that the boundary conditions 
given in \eqref{Unbdy} are small. In that case,
since we would have 
$a_n\cosh(y) + b_n\sinh(y) = o(1)$ for $y= \pm\f{L}{2}$
from \eqref{UnRep}, we easily see the statement of the lemma.

There are two types of terms
in the expression \eqref{Unbdy}: those given by $|(\im \mu)(p_n)|$ and 
$|(\im \mu)(q_n)|$, and those given by $V_n'(p_n)$ and 
$V_n'(q_n)$.  We will treat them separately.

We first claim that $|(\im \mu)(p_n)|$ and 
$|(\im \mu)(q_n)|$ decay as $n \to \infty$. This of course
easily follows from Lemma~\ref{lemma:mudecay}.

It is only slightly more difficult to use \eqref{Vnprimebdy}
to show that $V_n'(p_n)$ and 
$V_n'(q_n)$ decay to zero as $n \to \infty$. To see this,
note that for any fixed $y^*$, we have 
$\int_{y^*}^{\f{L_n}{2}} \f{\sinh(\f{L_n}{2}-s)}{\sinh(L_n)} ds =
O(e^{-\f{L_n}{2}-y^*})$, and that
$|(\im \mu)'|$ is bounded on all of the geodesic arc $\a_0$, while 
vanishing into the cusp by Lemma~\ref{lemma:mudecay}.  (Note also that 
$\int_{-\f{L_n}{2}}^{\f{L_n}{2}} \f{\sinh(\f{L_n}{2}-s)}{\sinh(L_n)}
ds \le 1$.) These preliminaries are enough to estimate $V_n'(p_n)$
(with an analogous argument for $V_n'(q_n)$).

We begin from \eqref{Vnprimebdy} with 
\begin{align*}
|V_n'(p_n)| &\le  \int_{-\f{L_n}{2}}^{y^*} \f{\sinh(\f{L_n}{2}-s)}{\sinh(L_n)}|(\im
\mu)'(\a_n(s))| ds \\ \notag
&\hskip1cm+ \int_{y^*}^{\f{L_n}{2}} \f{\sinh(\f{L_n}{2}-s)}{\sinh(L_n)}|(\im
\mu)'(\a_n(s))| ds \\
&\le \max_{[p_n, y^*]}|(\im \mu)'|\int_{-\f{L_n}{2}}^{y^*}
  \f{\sinh(\f{L_n}{2}-s)}{\sinh(L_n)}ds \\ \notag
&\hskip1cm+ \max_{\a_0}|(\im
  \mu)'|\int_{y^*}^{\f{L_n}{2}} \f{\sinh(\f{L_n}{2}-s)}{\sinh(L_n)}ds.
\end{align*}
Then, for $\e$ small, use Lemma~\ref{lemma:mudecay} to pick
$y^*$ so that $\max_{[p_n, y^*]}|(\im \mu)'| \le \e$. Then the first
term is bounded by $\e$ while the second term is bounded by 
$\max_{\a_0}|(\im \mu)'|O(e^{-\f{L_n}{2}-y^*})$. Letting
$L_n \to \infty$ and $-\f{L_n}{2}\le y^* \to -\infty$ somewhat more
slowly than $-\f{L_n}{2} \to -\infty$ (e.g. $y^* = -\f{L_n}{4}$)
shows that $V_n'(p_n) \to 0$.
\end{proof}

With these preliminaries, we easily conclude the proof of the 
proposition.  As $L_n \to \infty$, we find that for fixed $y$,
the kernels $K_n(y,s)$ limit on
$$
K(y,s)=\begin{cases} -\f{e^{s-y}}{2}, &s\le y \\
-\f{e^{y-s}}{2}, &y\le s
\end{cases}
$$
which we can write succinctly as
$$
K(y,s)=-\f{1}{2}e^{-d(y,s)}
$$
As both $(\im \mu)(s)$ and $(\im\mu)'(s)$ decay while 
$K_n(y,s)$ converge to an integrable function, we see that 
the formulas \eqref{VnRep} and \eqref{UnRep}
show uniform convergence of $V_n$ to a well-defined variation
field $V$, as well as primitives $U_n$ to a well-defined
(primitive) function $U$. As $U_n' = V_n$, the uniform convergence of 
$\{U_n\}$ and $\{V_n\}$ show that $U'=V$. Using
Lemma~\ref{anbndecay}, we obtain the 
formula
\begin{align} \label{URep}
U(y)&= \int_{p}^{q} K(y,s) (\im \mu)(s) ds.\notag \\
&= \int_{p}^{q} \f{1}{2}e^{-d(y,s)}(\im \mu)(s) ds
\end{align}
Next, it is easy to see that $U(y) \to 0$
as either $y \to p$ or $y \to q$. 
Mimicking the argument that showed that 
$V_n'(p_n) \to 0$, choose
$y^*$ close enough to $p$ so that 
$\max_{[p,y^*]}|\mu| < \f{\e}{2}$.  Then choose
$y$ even deeper into the cusp so that 
$d(y, y^*) \ge -\log (\f{\e}{\max_{\a_0}|\mu|})$
Then
\begin{align*}
|U(y)) &\le |\int_{p}^{y^*} K(y,s) (\im \mu)(s) ds| +|\int_{y^*}^{q}
 K(y,s) (\im \mu)(s) ds| \\
&<\f{\e}{2}\int_{p}^{y^*} |K(y,s)| ds +
 \max_{\a_0}|\mu|\int_{y^*}^{\infty} \f{1}{2} e^{-d(y,s)}ds\\
&\le \f{\e}{2}+ \f{1}{2} e^{-d(y,y^*)}\max_{\a_0}|\mu| \\
&\le \e
\end{align*}
as desired.

Finally, we address the finiteness of the energy of $U$.  Of course,
the metric $g_0$ near the cusp point $p$ (or $q$) may be expressed as 
$g_0 = |z|^{-2}(\log|z|)^{-2}|dz|^2$, and so the estimates for 
$\mu$ may be written as $|\mu(y)| = O(e^{y-e^y})$.  It is then an
easy estimate that
$$
\int_{\a_0} U'^2 + U^2 = \f{1}{2}\iint_{{\a_0}\times {\a_0}} e^{-|s-y|}\im
\mu(s) \im \mu(y) ds dy < \infty
$$ 
and is the limit of $\int_{\a_{0,n}} U'^2 + U^2$.

This concludes the proof of the Proposition.
\end{proof}

We summarize our discussion in this section with formulae for
the second variations of an open arc ${\a}$ analogous to those in 
Theorem~\ref{theorem:HessianFormula} for the second
variations of length of a simple closed curve $\g$.

\begin{thm} \label{theorem:ArcHessianFormula}
Along the Weil-Petersson geodesic arc
$\G(t)$, the second variation $\f{d^2}{dt^2}\ell$ of the $\G(t)$-length 
$\ell(t)=L(\G(t),{\a})$ of a (class of an) arc ${\a}$
is given by the (convergent) expression
\begin{align}\label{formula:Arcd2L}
\f{d^2}{dt^2}\ell(t) &= \int_{{\a}}- (\Delta - 2)^{-1}\f{2|\Phi|^2}{g^2_0} ds 
+ \int_{{\a}}
[U^{\Phi}_y]^2 +[U^{\Phi}]^2 ds\\
&=\int_{{\a}}- (\Delta - 2)^{-1}\f{2|\Phi|^2}{g^2_0} ds 
+ \f{1}{2}\iint_{{\a_0}\times {\a_0}} e^{-d(s,y)}\im
\mu(s) \im \mu(y) ds dy.
\end{align}

More generally, the Weil-Petersson Hessian $\text{Hess} L[\f{\bar\Phi}{g_0},
  \f{bar\Psi}{g_0}]$
is given by the (convergent) expression

\begin{align}\label{formula:ArcHessL}
\text{Hess} L[\f{\bar\Phi}{g_0},\f{\bar\Psi}{g_0}] &=
\int_{{\a}}- (\Delta - 2)^{-1}\f{2\re\Phi\bar\Psi}{g^2_0} ds + \int_{{\a}}
U^{\Phi}_yU^{\Psi}_y +U^{\Phi}U^{\Psi} ds\\
&=\int_{{\a}}- (\Delta - 2)^{-1}\f{2\re\Phi\bar\Psi}{g^2_0} ds\\
\notag 
&\hskip1cm +\f{1}{2}\iint_{{\a_0}\times {\a_0}}e^{-d(s,y)}\im
\mu(s) \im \nu(y) ds dy,
\end{align}
where $\mu=\f{\bar\Phi}{g_0}$ and $\nu=\f{\bar\Psi}{g_0}$ are the harmonic
Beltrami differential representatives of two tangent directions at 
$[g_0]$.
\end{thm}

Here $U^{\Phi}$ satisfies the equation \eqref{UPhiEquation} along the
arc $\g$, a condition which forces $U^{\Phi} \to 0$ along $\g$ as it
tends
to the punctures $p$ and $q$.

\section{Convexity for Laminations}\label{lamination estimates} 
We have already seen in
Theorem~\ref{theorem:HessianFormula} and 
Corollary~\ref{HessianEstimate} that if $\g$ is a simple closed curve, 
then on a Weil-Petersson ray $\G=\G(t)$, we have
\begin{equation}\label{curve estimate}
\f{d^2}{dt^2}\ell_\g(\G(t))\ge   
\frac{1}{3}\int_\g \|\Phi_t\|^2 ds
\end{equation} 
where $\Phi(t)$ is the holomorphic \qd\ tangent to $\G$ at $\G(t)$,
and $\|\Phi\|= \f{|\Phi|}{g_0}$. 
In this section, we extend that result to prove the

\begin{prop} \label{lamination hessian} Let $\G=\G(t)$ 
be a Weil-Petersson ray and $\la$ a measured lamination on $S$. Then
\begin{equation}\label{lamination estimate}
\f{d^2}{dt^2}\ell_\la(\G(t))\ge
\frac{1}{3}\int_\la \|\Phi_t\|^2 ds
\end{equation}
where $\Phi_t$ is the holomorphic \qd\ tangent to $\G$ at $\G(t)$.
\end{prop}

Our definition of $\int_\la\|\Phi\|^2ds$ is straightforward and
parallels the definition of length of the lamination $\la$. See
\cite{bonahon:contempmath2001} for a background discussion. In particular, a
measured lamination $\la$ is defined as a measure 
$\la(k)= \int_k d\lambda$ on
transverse arcs $k$. Choose arcs $k_1,\dots,k_J$ which are transverse
to $\la$ and construct flow boxes $\{F_i\}$ for $\la$ 
bounded by the $k_j$ and parallel to $\la$.

Then if $G_\la$ is the geodesic lamination underlying $\la$, then
$\la-\cup_jk_j$ is a (possibly infinite) 
collection of finite length arcs. The length of
$\la$ is then the integral, with respect to the transverse measure
$d\la$ of $\la$, of the lengths of the
finite arc components. More precisely, we lift each of these
components $\la$ to $\hat\la_a\subset T^1M$, endow $\hat\la_a$ with
the natural arclength measure $ds$, and then 
integrate the product to get
$$
\ell(\la) =: \int_\la ds =: \iint_{\la_a} ds d\la(a).
$$
In order to define $\int_\la\|\Phi\|^2ds$, we proceed analogously,
except that we note that the function $\|\Phi\|^2$ on $S$ then
naturally defines 
a measure $\|\Phi\|^2ds$ on $T^1M$. In other words, we set
$$
\int_\la \|\Phi\|^2ds = \iint_{\la_a} \|\Phi\|^2ds d\la(a).
$$

\begin{proof}[Proof of Proposition~\ref{lamination hessian}] 
Let $\g_n$ be a sequence of simple closed curves
  converging to $\la$. The idea is to apply \eqref{curve estimate} to
  $\g_n$ and 
then take a limit in $n$ to find \eqref{lamination estimate}.

Now $\ell_\g$ and $\ell_\la$ are real analytic functions on $\G$ 
\cite{kerckhoff:earthquakesanalytic85},
and thus since $\ell_{\g_n}\to\ell_\la$, so does
$\f{d^2}{dt^2}\ell_{\g_n}\to\f{d^2}{dt^2}\ell_\la$. Thus the left-hand
sides of \eqref{curve estimate} 
converge to the left-hand side of  \eqref{lamination estimate}.

For the right-hand side, the argument is virtually tautological. We
first note that the arclength measure $\|\Phi\|^2ds$ is continuous on
$T^1M$. Then, choose $n$ sufficiently large so that the flow boxes
$\{F_i\}$ described above also serve as flow boxes for $\g_n$. Then
the definition of 
convergence in $\SM\SL$ then easily 
implies that the right-hand side of  \eqref{curve estimate} 
converges to the right-hand side of  \eqref{lamination estimate}.
\end{proof}
 
\section{Convexity of $\ell^{\f12}$ and upper bounds on the
  Hessian}\label{sec:Two-thirds convex}

\subsection{The function $\ell^{\f12}$.}\label{sec:One-half convex}

\subsubsection{The function $\ell^{\f23}$.}
In light of \eqref{curve estimate}, we can quickly refine the
basic convexity result for the length $\ell$ to a convexity result 
for a concave function of $\ell$, namely $\ell^{\f23}$. We will see in
the next subsection that this is not sharp, but at this stage it is 
elementary.

To begin, recall from \eqref{FirstVariation} that the first 
variation of length may be expressed as

\begin{equation*}
\f{d}{dt} \ell_{\g}(\G(t)) =\int_{\g_0} \f{\ree \Phi}{g_0} ds
\end{equation*}

We can then estimate this derivative as 
\begin{align*} 
\bigl|\f{d}{dt} \ell_{\g}(\G(t))\bigr| &\le \int_{\g_0} 
\bigl|\f{\ree \Phi}{g_0}\bigr| ds\\
&\le\big(\int_{\g_0}\f{|\Phi|^2}{g_0^2}ds\big)^{\f12}\ell_{\g_0}^{\f12}
\end{align*}
Squaring and combining with \eqref{curve estimate} yields
\begin{equation}
\ell\f{d^2}{dt^2}\ell_{\g}(\G(t)) \ge \f13\big(\f{d}{dt}
\ell_{\g}(\G(t))\big)^2.
\end{equation}
We compute $\f{d^2}{dt^2} \ell^{\f23}(\G_t)$ and substitute in the
above inequality to conclude the 
\begin{cor} \label{corollary:two-thirds convex}
The function $\ell^{\f23}$ is Weil-Petersson convex on \Tec space.
\end{cor}

\subsubsection{The length of an annulus.}\label{annulus}

In this subsection, we compute in two ways the second variation of the 
length of the core geodesic in a hyperbolic annulus; the basic result
is well-known (see Example~3.6 in \cite{wolpert:behavior08}).

let $\SC(\ell)$ denote a complete hyperbolic annulus whose core
geodesic has length $\ell$.
This cylinder may be parametrized as
$$\SC(\ell) = [-\f{\pi}{2\ell},\f{\pi}{2\ell}] \times [0,1],$$
where top and bottom edges are identified, and we consider the 
metric cylinder as equipped with the hyperbolic metric 
$g=ds_{\ell}^2 = \ell^2\csc^2 \ell x |dz|^2$. 

Consider rotationally the harmonic map
$R(t): \SC(\ell)\to \SC(\ell+t)$ which does not twist the boundary;
in other words, this map may be expressed in coordinates
as $R(t)= u(t) + i v(t)$ where $u(t)(z) = u(t)(x)$ and
$v(t)(z) = y$. Now the rotationally invariant 
holomorphic quadratic differentials
on a cylinder have a particularly simple form: we may write
one as $\Phi = c dz^2$ in the complex coordinates above.
Thus, taking one of these as a Hopf differential for our map,
and using that the holomorphic energy $\SH$, the Beltrami
differential $\nu$ and the Hopf differential $\Phi$ may be related 
by $\Phi = g(\ell_0)\SH\bar\nu$, we conclude that
\begin{equation*}
u' = \f{1+\nu}{1-\nu}=\f{\ell^2 + c \cos^2\ell x}{\ell^2 - c
  \cos^2\ell x}.
\end{equation*}

If $c(t)$ is the factor so that the Hopf differential 
is parametrizing a family of harmonic maps 
$R(t): \SC(\ell) \to \SC(\ell +t)$ whose targets are 
progressing through \tec space $\ST(\SC)$ at unit
\WP speed, then two conditions hold: (i) the choice of $c(t)$
provides for $c(t)dz^2$ to the Hopf differential for the
map $R(t)$, and (ii) $\|\f{d}{dt}\Phi(t)\|_{WP} = 1$.

Now, for $w(t)$ to have image $\SC(\ell +t)$, we must have the
boundary of $\SC(\ell)$ map to the boundary of $\SC(\ell +t)$, i.e.
\begin{equation*}
\f{\pi}{2(\ell + t)}= u(\f{\pi}{2\ell}) = \int_0^{\f{\pi}{2\ell}}
u'(x)dx +u(0) = \int_0^{\f{\pi}{2\ell}}\f{1+\nu}{1-\nu}=\f{\ell^2 
+ c \cos^2\ell x}{\ell^2 - c
  \cos^2\ell x}dx.
\end{equation*}
Upon differentiating in in $t$ and finding the resulting elementary
integrals, we obtain $\dot c = -\ell$.  Thus
\begin{equation*}
\|\f{\partial}{\partial \ell}\|^2_{WP} = \|\dot c
g_{\ell}^{-1}\|_{WP}= \|(-\ell) \ell^{-2}\cos^2\ell
x\|_{WP}=\f{\pi^2}{\ell}
\end{equation*}
after another explicit integration. Thus 
$\|\f{\partial}{\partial \ell}\|_{WP}= \f{\pi}{\ell^{\f12}}$
and so $ds^2_{WP} =\f{\pi^2}{\ell}d\ell^2$ on the 
\tec space $\ST(\SC)$.  This implies that on this space,
$\f{d\ell}{ds}= \pi^{-1}\ell^{\f12}$ and so 
$\ell = (2\pi)^{-2}s^2$.  Thus the length $\ell$ of the core geodesic
satisfies that $\ell^{\f12}$ is convex, but not convex to any
lower power.

\begin{remark} Alternatively, we may use the formulas
  \eqref{formula:d2L} and \eqref{FirstVariation}
to analytically find the same result.  In that case, if we set
$\Phi=c dz^2$, then 
$\ddot\ell = \int_{\g_0} -2(\Delta-2)^{-1}\{c^2 g^{-4}\}$. However,
the equation $\Delta u-2u = c^2 g^{-4}$ reduces in this case to an 
ordinary differential equation, whose solution we can require to be 
bounded on the boundary. It is then elementary (see the analogous
analysis in \eqref{BVP}ff using the method of variations of 
parameters) to find an exact expression for $u$ on $\g_0$.  
We then compare with the expression in \eqref{FirstVariation}
for the first derivative to to obtain the half-power
convexity result above: this confirms, at least in this
very simple case, the formula \eqref{formula:d2L}.
\end{remark}

\subsubsection{$\ell^{\f12}$ is convex.}
We now offer a geometric argument of Wolpert's recent result
\cite{wolpert:behavior08} that $\ell_{\g}^{\f12}$ is 
\WP convex. 

\noindent\emph{Comparison of lifted harmonic map to rotationally invariant
  map.} The essential point is best understood in the setting of the 
annular covers $(\SC, \tilde{g_t})$ of the family of surfaces
$(S, g_t)$. Consider the harmonic maps $w_t: (S, g_0) \to (S, g_t)$
and their lifts $\tilde{w_t}: (S, \tilde{g}_0) \to (S, \tilde{g}_t)$.
These lifts are in the homotopy class of the rotationally invariant
harmonic map $R_t: \SC(\ell_0) \to \SC(\ell_t)$, where
$\ell_{\g}(g_t) = \ell_t$.  Now $w_t$ is conformal only at the 
(isolated) zeroes of the Hopf differential, and so, off of small
neighborhoods of the zeroes of the lifted Hopf differential,
the harmonic map $\tilde{w_t}$ has quasi-isometric constant
uniformly bounded away from $1$. By contrast, one can
either compute or reason geometrically that the rotationally
invariant harmonic map $R_t: \SC(\ell_0) \to \SC(\ell_t)$
has quasi-isometric constant tending uniformly to $1$ as 
one leaves compacta in $\SC(\ell_0)$:  the image curves are growing
exponentially in length, so energy efficiency requires the map to be 
increasingly close to an isometry as one leaves compact sets.

 Let $\SH^R(t)$ be the holomorphic energy (see
equations \eqref{pullback}-\eqref{Hdotdot}) of $R(t)$
and $\SH(t)$ be the holomorphic energy of $\tilde{w_t}$. 
Since both $R(t)$ and $\tilde{w_t}$ are the identity when
$t=0$, and since, as we have just seen, $R(t)$ is asymptotically
an isometry while $\tilde{w_t}$ is boundedly away from being the 
identity off small sets, then we find
\begin{equation*}
\ddot{\SH}(t) \ge \ddot{\SH}^R(t)
\end{equation*}
outside some large compact set (at least away from small neighborhoods of
the zeroes of the lift $\tilde{\Phi}$ of $\Phi$). In particular,
parametrizing $\SC(\ell)$ as in subsection~\ref{annulus}, we see that
\begin{equation}\label{ODEbdy}
\int_{x=\pm\f{\pi}{2\ell}\mp\d}\ddot{\SH}(t) \ge 
\int_{x=\pm\f{\pi}{2\ell}\mp\d}\ddot{\SH}^R(t).
\end{equation}

\noindent\emph{A comparison of ODEs.} The rest of the proof follows 
by applying other inequalities that reflectthat $R(t)$
is a harmonic map of lower (regularized in some way) energy that
$\tilde{w_t}$. In particular, consider the Fourier expansion
$\Phi = \sum b_n(x)e^{2\pi iny}$ of the quadratic differential
$\Phi$ (where the map $\tilde{w_t}$ has Hopf differential $\Phi$).

Because 
\begin{equation*}
\dot{\ell}= \int_{x=0}\f{\ree \Phi}{g}\sqrt{g}dy = \ell \ree b_0
\end{equation*}
we know that the Hopf differential $\Phi^R$ for the 
rotationally invariant map $R(t)$ must be
$\Phi^R = \ree b_0 dz^2$ on $\SC(\ell)$: this is because
the targets $\SC(\ell + t)$ agree for the two maps $w_t$ and $R(t)$,
and hence the change in core-curve length is the same.

The upshot is that, for an arbitrary constant curvature
circle $\{x=\xi\}$, we have
\begin{equation} \label{ODERHS}
\int_{x=\xi} \f{|\Phi|^2}{g^2}ds = g^{-\f32}(x)\sum|b_n|^2 
\ge g^{-\f32}(x)(\ree b_0)^2 = \int_{x=\xi} \f{|\Phi^R|^2}{g^2}ds.
\end{equation}

Of course, we know from formula \eqref{formula:d2L} that
\begin{equation*}
\ddot{\ell} \ge \int_{\g_0} -2(\Delta-2)^{-1}\f{|\Phi|^2}{g^2}ds 
=\f12 \int_{\g_0}\ddot{\SH}ds
\end{equation*}
 
To estimate this last integral, let $u$ be the solution of 
\begin{equation*}
\Delta_g u-2u = \f{-2|\Phi|^2}{g^2}.
\end{equation*}
If we were to integrate this equation along the vertical parameter
curves $\{x=const\}$, we would obtain an ordinary 
differential equation for the function 
$\int_x u dy = \int_x \f12\ddot{\SH} dy$ in the single variable
$x \in (-\f{\pi}{2\ell},\f{\pi}{2\ell})$, i.e.
\begin{equation*} \label{ODE-Hdotdot}
\f1g\partial_x^2 \int_x u dy - 2\int_x u dy = -2\int_x
\f{|\Phi|^2}{g^2}dy
\end{equation*}
Of course, a similar equation holds for the integrals 
$\int_x \f12\ddot{\SH}^R dy$ and $\int_x \f{|\Phi^R|^2}{g^2}dy$
associated to the rotationally invariant map. Indeed,
inequality \eqref{ODEbdy} says that there are boundary
points $x=\pm\f{\pi}{2\ell}\mp\d$ at which 
$\int_x \f12\ddot{\SH}^R dy \le \int_x \f12\ddot{\SH} dy$; 
moreover, inequality \eqref{ODERHS} asserts that on the interval
$(-\f{\pi}{2\ell},\f{\pi}{2\ell})$, the right-hand-side
of \eqref{ODE-Hdotdot} is less in the lifted case than
in the rotationally invariant case.

The upshot is that the comparison principle for ordinary differential
equations implies that 
\begin{align*}
\int_x -2(\Delta-2)^{-1}\f{|\Phi|^2}{g^2}dy 
&= \int_x \f12\ddot{\SH}^Rdy \\
&\ge \int_x \f12\ddot{\SH} dy \\
&= \int_x -2(\Delta-2)^{-1}\f{(\ree b_0)^2}{g^2}dy.
\end{align*}
In particular, specializing to the curve $\{x=0\}$, and recalling the
implication above of \eqref{formula:d2L}, we find
\begin{equation*}
\ddot{\ell} 
\ge \int_{\g_0} -2(\Delta-2)^{-1}\f{|\Phi|^2}{g^2}
\ge \int_{\g_0} -2(\Delta-2)^{-1}\f{(\ree b_0)^2}{g^2}
= \ddot{\ell^R}
\ge 2\f{\dot{\ell}^2}{\ell}.
\end{equation*}
Here the last inequality is inherited from the rotationally
invariant case: recall that our choice that 
$\ree \int_{\g_0} \Phi = \int_{\g_0} \Phi^R$ implies that
the infinitesimal change of lengths agree between the 
lifted and rotationally invariant maps.  We conclude
\begin{cor}
(Wolpert \cite{wolpert:behavior08}) The function $\ell^{\f12}$
is \WP convex in the \tec space $\ST(S)$.
\end{cor}

\subsection{A general upper bound for the Hessian}

We have already seen a lower bound for the Hessian of length
in Corollary~\ref{HessianEstimate} and 
Proposition~\ref{lamination hessian}. In this passage, we note an 
easy upper bound as well. 

We begin by noting that if 
\begin{equation*}
(\Delta-2)h = -2\f{|\Phi|^2}{g_0^2}
\end{equation*}
on a surface $S$, then the maximum principle implies 
\begin{equation}\label{maxprinciplebound}
h \le \| \f{|\Phi|^2}{g_0^2} \|_{\infty}
\end{equation}
where the right hand side is the maximum of the function 
$\f{|\Phi|^2}{g_0^2}$ on $S$. In the formula~\eqref{formula:d2L},
this will estimate the first term.

To estimate the second term, we consider equation~\eqref{primitive}
(combined with \eqref{FDefined})
\begin{equation}
U_{yy} -U = -\f{\im\Phi}{g_0}.
\end{equation}
The maximum principle then implies that 
\begin{equation}
U \le \max_{\g}\big|\f{\im\Phi}{g_0}\big|.
\end{equation}
Thus the second term in formula~\eqref{formula:d2L}
is estimated as 
\begin{align}\label{EnergyBoundByMax}
\int_{\g} U^2_y + U^2 &= - \int_{\g} (U_{yy} - U)U\\ \notag
&= \int_{\g} (\f{\im\Phi}{g_0})U\\ \notag
&\le \int_{\g}\big(\max_{\g}\big|\f{\im\Phi}{g_0}\big|\big)\big(\max_{\g}\big|\f{\im\Phi}{g_0}\big|\big)
\quad\text{after a substitution }\\ \notag
&= \ell_{\g}(\max_{\g}\big|\f{\im\Phi}{g_0}\big|)^2.
\end{align}
We conclude, taking into account Corollary~\ref{HessianEstimate}, that
\begin{cor}
\begin{equation}
\frac{1}{3}\int_\g \|\Phi_t\|^2 ds \le \f{d^2}{dt^2}\ell_\g(\G(t))\le 
\ell_{\g}(\max_S  \| \f{|\Phi|^2}{g_0^2} \| +
(\max_{\g}\big|\f{\im\Phi}{g_0}\big|)^2).
\end{equation}
\end{cor}

\subsection{Estimates for the Weil-Petersson connection near the 
compactification divisor}

In this passage we refine the method above to estimate the
Weil-Petersson connection on a codimension two distribution
$\SP \subset T\SM$ of the tangent bundle near the Deligne-Mumford
compactification 
divisor (i.e. for
surfaces with small injectivity radius) which is in some sense 
``parallel'' to the tangent bundle of the compactification 
divisor.  Roughly, we prove that $\SP$ is quite flat, in the sense
that for $X,Y \in \SP$, we will have that the normal component
$(\nabla_X Y)^{\bot}$ of $(\nabla_X Y)$ satisfies 
$(\nabla_X Y)^{\bot} = O(\ell^2)$.  (Here $\ell$ signifies the length
of the curve which vanishes on the nearby component of the
compactification
divisor.

To state this precisely, we choose a simple closed curve $\g \subset
S$ and a small number $\ell >0$; we consider the level set
$L_{\g}(\ell)$ of hyperbolic surfaces for which $L(\cdot, \g)=\ell$.
The set $L_{\g}(\ell)$ is a submanifold of the \Tec space $\ST$ of 
real codimension one, and it is orthogonal to the vector 
$\grad \ell_{\g}$, the \WP gradient of $\ell_{\g}$. Let $J$ be the
almost complex structure of $\ST$ and consider the projection 
$\tau \in TL_{\g}(\ell)$ of $J \grad \ell_{\g}$ into 
$ TL_{\g}(\ell)$.  Let $\SP \subset TL_{\g}(\ell)$ denote the
distribution of $(\dim \ST -2)$-planes in $TL_{\g}(\ell)$
orthogonal to the span of $\grad \ell_{\g}$ and $\tau$;
note that $\SP$ is \WP orthogonal to both  $\grad \ell_{\g}$ and
$J\grad \ell_{\g}$.

\begin{remark}
One does not expect $\SP$ to be integrable. In particular,
one does not expect that $J\grad \ell_{\g}$ is parallel to the 
Fenchel-Nielsen twist vector field. (See \cite{wolpert:FNtwist82}.)
Nevertheless, 
it will be a consequence of equation \eqref{formula:NoPeriod} of the
early part of the next proof that, as $\ell_{\g} \to 0$, the 
distribution $\SP$ converges to the tangent bundle $T\SC_{\g}$ of the
compactification divisor $\SC_{\g} = \{\ell_{\g}=0\}$ of the 
augmented \Tec space $\bar{\ST}$. (Compare \cite{Wo91a}.)
\end{remark}

Now, for $X \in \SP$ and $Y$ a section of $\SP \to \ST$, (i.e. a 
vector field on $L_{\g}(\ell)$), we can consider the \WP 
covariant derivative $\nabla_X Y$.  Of course, the vector
 $\nabla_X Y$ has components both in $\SP$ and in the orthogonal
complement $\SP^\bot$ of $\SP$; we focus here on the component of
this vector in the 
orthogonal complement.

\begin{remark}
It might be difficult to formulate general results on  the full 
vector $\nabla_X Y$ in useful and incisive way.  For example, if we were
to ``lift'' a curve $\a$ from the compactification divisor 
$\SC_{\g}$ to an almost parallel curve $\hat\a$ tangent to $\SP$,
then since $\nabla_{\dot\a}{\dot\a}\subset T\SC_{\g}$ can be
arbitrary, so might we expect $\nabla_{\dot{\hat\a}}{\dot{\hat\a}}$
to have an arbitrary non-orthogonal component.
\end{remark}

Our main result in this section is
\begin{thm}\label{FlatParallel} In the notation above, for $X\in \SP$ 
and $Y\subset \SP$ of unit norm, we have
\begin{equation}
(\nabla_X Y)^\bot = <\nabla_X Y, \SP^{\bot}> = O(\ell_{\g}^2).
\end{equation}

In particular, $<\nabla_X Y,\grad \ell_{\g}> =  O(\ell_{\g}^2)$ and
$<\nabla_X Y,J\grad \ell_{\g}> =  O(\ell_{\g}^2)$.
\end{thm}

\begin{remark}
Similar results were obtained recently by Wolpert \cite{wolpert:behavior08},
using his estimates on the Hessian.  From the formula for the Hessian
presented in \eqref{formula:d2L}, we see in the present derivation
the elementary nature of the expansion of $\nabla_X Y$ in $\ell_{\g}$.
The flatness of order $O(\ell_{\g}^2)$ occurs because the (explicit)
rotationally invariant even solution $u(x)= x \tan \ell x + 1/\ell$
of the Jacobi equation
has the following property: at the core geodesic of the cylinder, 
this function $u(x)$ is 
smaller by a factor comparable to $O(\ell_{\g})$ than it is on the 
boundary of the cylinder.  Since the geodesic has length
$O(\ell_{\g})$, the dominant term of the integral of $u(x)$ over the
geodesic decays like $O(\ell_{\g}^2)$.
\end{remark}

\begin{proof}
We begin the proof of Theorem~\ref{FlatParallel} with a preliminary 
proposition characterizing the quadratic differentials which represent 
elements of $\SP$.

\begin{prop} \label{NoPeriod} 
Let $X \in T_{[M,g]}\ST$ with $X\in\SP$ and $\Phi=\Phi_X$ be a 
holomorphic quadratic differential on the surface $(S,g)$ for which 
$\bar\Phi/g$ is a harmonic Beltrami differential representing $X$. 
Then for $\g$ the geodesic on $S$ used to define $L_{\g}(\ell)$,
we have 
\begin{equation}\label{formula:NoPeriod}
\int_{\g}\f{\Phi}{g}ds =0.
\end{equation}
\end{prop}
\begin{proof}
Since $X \in \SP$, we have $<X, \grad \ell_{\g}>=0$. But from
\eqref{FirstVariation}, we find that  
\begin{align*}
0&= <X, \grad \ell_{\g}>\\
&=X(\ell_{\g})\\
&=\int_{\g}\f{\re\Phi}{g}ds.
\end{align*}
Of course the other defining condition of $X \in \SP$ is that 
$<X, J\grad \ell_{\g}> =0$.  But since the \WP metric is K\"ahler,
this implies
\begin{align*}
0&=<X, J\grad \ell_{\g}>\\
&=-<JX, \grad \ell_{\g}>.
\end{align*}

But $JX$ is represented by $i\bar\Phi/g$, and so we derive as above
that 
\begin{equation*}
0=\int_{\g}\f{\re i\Phi}{g}ds,
\end{equation*}
proving the result.

\end{proof}

\begin{remark}
The covectors in $T^*\SC_{\g}$ are represented by holomorphic
quadratic
differentials which have simple poles at the nodes. On
the other hand, a generic covector in $T_{\Sigma}^*\bar{\ST}$
at an element $\Sigma \in \SC_{\g}$ has 
a second order pole with a non-vanishing residue at the node obtained 
by pinching $\g$.
\end{remark}

\emph{Continuation of the Proof of Theorem~\ref{FlatParallel}} The heart of the
matter is a computation of $\text{Hess} \ell_{\g}(X,X)$, in particular to 
prove

\begin{lem} \label{FlatnessLemma}
For $X \subset \SP$ as above, we have 
\begin{equation}
\text{Hess} \ell_{\g}(X,X) = O(\ell_{\g}^2).
\end{equation}
\end{lem}
To see that this is enough, note first that polarization will imply
then that $\text{Hess} \ell_{\g}(X,Y) = O(\ell_{\g}^2)$ for $X,Y \subset \SP$
of unit norm. But then
\begin{align}
-<\nabla_XY, \grad \ell_{\g}> &= -(\nabla_XY)\ell_{\g}\notag \\
&=(XY-\nabla_XY)\ell_{\g} \qquad\text{since $X,Y$ are tangent to
  $L_{\g}(\ell)$} \notag \\
&=\text{Hess}\ell_{\g}(X,Y) \qquad\text{by definition} \notag\\
&= O(\ell_{\g}^2).\label{ConnectionAlmostFlat} \qquad\text{by
  Lemma~\ref{FlatnessLemma}}
\end{align}
Moreover
\begin{align*}
<\nabla_XY, J\grad \ell_{\g}> &= -<J\nabla_XY, \grad \ell_{\g}> \\
&\hskip1cm \qquad\text{as $J$ is a \WP isometry} \\
&=-<\nabla_XJY, \grad \ell_{\g}>\\ 
&\hskip1cm\qquad\text{as \WP is a
  K\"ahler metric}\\
&= -<\nabla_XZ, \grad \ell_{\g}>
\end{align*}
for some $Z\in \SP$ as $\SP$ is $J$-invariant, being orthogonal to a 
$J$-invariant subspace of a K\"ahler manifold. Then
$<\nabla_XY, J\grad \ell_{\g}>= O(\ell_{\g}^2)$ follows in the
manner of \eqref{ConnectionAlmostFlat}.  This concludes the 
proof of Theorem~\ref{FlatParallel}, pending the proof of the main
lemma.
\end{proof}

\begin{proof}[Proof of Lemma~\ref{FlatnessLemma}]
The basic idea of the proof is to estimate the terms in the
formula \eqref{formula:d2L} for the Hessian, where we 
take $X$ to be represented by a harmonic Beltrami 
differential $\mu = \bar\Phi/g$, and we apply the features
of $X \subset \SP \subset TL_{\g}(\ell)$ to prove that those
terms are small. In particular, because the length
$L(g,\g)$ is small, the geodesic $\g$ is embedded in a wide, 
thin collar. Then, by Proposition~\ref{NoPeriod}, we
learn that $|\Phi|$ must decay rapidly towards the 
center of the collar.  Those facts together are 
enough to conclude that each of the pair of terms 
in \eqref{formula:d2L} is small.

We carry out the plan in steps.

\paragraph{Step 0. The collar.} Consider the collar 
$\SC = [-\f1{\ell}\sec^{-1}\f1{\ell}, \f1{\ell}\sec^{-1}\f1{\ell}]
\times [0,1]$ with horizontal edges 
$ [-\f1{\ell}\sec^{-1}\f1{\ell}, \f1{\ell}\sec^{-1}\f1{\ell}] \times
\{0,1\}$ identified, equipped with the hyperbolic metric 
$g_0 = \ell^2\sec^2\ell x |dz^2|$.  This collar $\SC$ embeds 
in a neighborhood of the geodesic $\g$, with 
$\{0\} \times [0,1]$ mapping onto $\g$.

\paragraph{Step 1. Decay of $|\Phi|$} On this annular collar
$\SC$, we may regard the quadratic differential $\Phi$ as a function
(or more formally, we divide $\Phi$ by the nonvanishing holomorphic 
quadratic differential $dz^2$ to obtain a function in the 
quotient). Then, by the rotational invariance of the collar
(or by working in Fermi coordinates), we see that $g$ may
be taken as constant along $\g$, and so the conditions
in Proposition~\ref{NoPeriod} imply that $\Phi$ has no period in
the collar.  Finally, we estimate boundary conditions.

Since, for $\ell$ small, every $X\subset \SP \subset TL_{\g}(\ell)$
can be approximated on compacta away from $\g$ by an
integrable  meromorphic
quadratic differential on a noded Riemann surface in the
compactification divisor $\SC_{\g}$, and the collection of such 
quadratic differentials of unit \WP norm is compact, we see that 
we may take $|\Phi|$ as bounded on the horocycle of length
one on $\partial\SC$.

Because $\Phi$ is holomorphic and hence harmonic, 
the vanishing of the period together
with the fixed boundary conditions is enough to show that 
$|\Phi|$ decays rapidly on the interior of the cylinder.

We can obtain an estimate of this decay through a Fourier 
analysis of $\Phi$. On the cylinder $\SC$, set 
$\Phi = \sum a_n(x) e^{2\pi i ny}$.  Then the harmonicity
of $\Phi$ implies
\begin{equation*}
0=\triangle \Phi = \sum (a_n''(x) - 4\pi n^2 a_n(x)) e^{2\pi i ny}
\end{equation*}
and in particular the equations
\begin{equation*}\label{andoubleprime}
a_n''(x) - 4\pi n^2 a_n(x)=0.
\end{equation*}
Now as $a_0(0)= \int_{x=0}\Phi =0$ by \eqref{formula:NoPeriod}, we find that 
$\int_{\g^*}\Phi =0$ along any cycle $\g^*$ in 
$\SC_{\g}$ homologous to $\g$, hence $a_0(x)=0$.

Of course, $\int_{\partial \SC_{\g}} |\Phi|^2 \le C_0$ by our argument
on limits above, and so
\begin{equation}\label{anbdy}
\sum_{n \ne 0} a_n^2(\pm \f{1}{\ell}\sec^{-1}\f{1}{\ell})= \int_{\partial
  \SC_{\g}} |\Phi|^2 \le C_0.
\end{equation}
We note that 
\begin{align}\label{andiffineq}
(\sum_{n \ne 0}a_n^2(x))''&=\sum_{n \ne 0}2a_n''a_n + 2(a_n')^2\\ \notag
&=\sum_{n \ne 0}8\pi^2n^2a_n^2 + 2 (a_n')^2 \qquad\text{by
    \eqref{andoubleprime}}\\ \notag
&\ge 8\pi^2\sum_{n \ne 0}a_n^2. 
\end{align} 
Thus by the maximum principle applied to the 
differential inequality \eqref{andiffineq} with 
boundary conditions \eqref{anbdy}, we have
\begin{equation}\label{PhiOnVerticalLines}
\int_{x=x_0} |\Phi|^2 \le \f{C_0 \cosh\sqrt{8}\pi
  x_0}{\cosh\sqrt{8}\pi(\f{1}{\ell}\sec^{-1}\f{1}{\ell})} :=D\cosh\sqrt{8}\pi
  x_0.
\end{equation}
Finally consider $\Phi(z_0)$, for $z_0 = x_0 + i y_0$,
a fixed point of the collar. Note that on our parametrization of the 
collar, the balls of radius $\f{1}{2}$ inject into the parameter domain.
As $\Phi$ is harmonic,
\begin{align*}
|\Phi(z_0)| &\le \f{4}{\pi}|\int_{B_{\f{1}{2}}(z_0)}\Phi|\\
&\le \f{4}{\sqrt{\pi}}(\int_{B_{\f{1}{2}}(z_0)}|\Phi|^2)^{\f12}\\
&< \f{4}{\sqrt{\pi}}(\int_{|\re(z-z_0)\le \f{1}{2}}|\Phi|^2)^{\f12}\\
&= \f{4}{\sqrt{\pi}}(\int_{\re z_0 -\f{1}{2}}^{\re z_0 +\f{1}{2}} \int_{x=t}
|\Phi|^2dydt)^{\f12}\\
&\le \f{4}{\sqrt{\pi}}(\int_{\re z_0 -\f{1}{2}}^{\re z_0 +\f{1}{2}} D\cosh\sqrt{8}\pi
t dt)^{\f12} \qquad\text{by \eqref{PhiOnVerticalLines}}\\
&=D_1(\cosh\sqrt{8}\pi x_0)^{\f12}\\
\end{align*}
We conclude that 
\begin{equation*}
|\Phi(z_0)|^2 < D_2\cosh\sqrt{8}\pi x_0,
\end{equation*}
where (using \eqref{PhiOnVerticalLines})
\begin{equation} \label{DEstimate}
D_2 = O(e^{-\f{\sqrt{8}\pi}{\ell}})
\end{equation}
in $\ell$, for $\ell$ small, justifying our remark
about the decay of $|\Phi(z)|$ into the collar.

\paragraph{Step 2: the function $U$.}
There are two terms in the formula \eqref{formula:d2L} for the Hessian;
here we estimate the one involving the energy of the function $U$.
In particular, we already know from \eqref{EnergyBoundByMax} that 
\begin{align*}
\int_{\g_0} U'^2 + U^2 &\le \ell(\max \f{|\Phi|}{g})^2\\
&\le \ell g^{-2}|_{x=0}D_2\\
&=O(\ell^{-3}e^{\f{-\sqrt{8}\pi}{\ell}})
\end{align*}
using $g|_{x=0} = \ell^2$ and \eqref{DEstimate}.  This term is then
consistent with the statement of the lemma that $\text{Hess}\ell(X,X) \le
O(\ell^2)$.

\paragraph{Step 3: the function $(\triangle
  -2)^{-1}\f{|\Phi|^2}{g^2}$.} We are left to estimate the second term
     in the expression \eqref{formula:d2L} for $\text{Hess}\ell(X,X)$, namely
\begin{equation*}
\int_{\g}-2(\triangle-2)^{-1}\f{|\Phi|^2}{g^2}.
\end{equation*}

In particular, we need to estimate the solution $u_0$ to the equation
\begin{equation}
(\triangle-2)u_0=-2\f{|\Phi|^2}{g^2}.
\end{equation}
evaluated on the core geodesic.

We again estimate the solution to this by estimating it on the 
cylinder $\SC_{\g}$, on which we have good control on
$\f{|\Phi|^2}{g^2}$. Of course we already know from the maximum
principle (see \eqref{maxprinciplebound}) that 
\begin{equation*}
u_0 \le \sup \f{|\Phi|^2}{g^2}.
\end{equation*}

Now the latter is bounded $\SC_{\g}$, and the complement of
 $\SC_{\g}$ has bounded geometry on which 
$\int _{M\sim\SC_{\g}}\f{|\Phi|^2}{g^2}dA_g \le 1$
(because $\Phi$ is of unit \WP norm).  We conclude
 that 
there is a $C_0$ for which 
\begin{equation*}
\f{|\Phi|^2}{g^2}\big|_{\partial\SC_{\g}} \le C_0
\end{equation*}
Thus by the maximum principle, it is enough to estimate the solution
$u$ of the boundary value problem

\begin{align*}
\f{1}{g}u''(x) - 2u(x) &= \f{D_2\cosh\sqrt{8}\pi x}{g^2}\\
u(\pm\f{1}{\ell}\sec^{-1}\f{1}{\ell}) &= C_0
\end{align*}
Using that $g=\ell^2\sec^2\ell x$, we rewrite the equation above as
\begin{align}\label{BVP}
u''(x) - 2\ell^2\sec^2\ell x u(x) &= D_3\cosh(\sqrt{8}\pi x)\cos^2\ell
x\\ \notag
u(\pm\f{1}{\ell}\sec^{-1}\f{1}{\ell}) &= C_0,
\end{align}
where $D_3 = O(\ell^{-2}e^{-\sqrt{8}\pi/{\ell}})$. The homogeneous 
equation
\begin{equation*}
u''(x) - 2\ell^2\sec^2\ell x u(x) = 0
\end{equation*}
has the two solutions
\begin{align}
u_1(x) &= \tan \ell x\\ \notag
u_2(x) &= x\tan \ell x + \f{1}{\ell}.
\end{align}
Using these, one can solve for a particular solution of the form
$u_0 = u_1v_1 + u_2v_2$ by elementary integrations, namely
\begin{align*}
v_1 &= \int u_2  D_3\cosh(\sqrt{8}\pi x)\cos^2\ell x\\ \notag
v_2 &= - \int u_1  D_3\cosh(\sqrt{8}\pi x)\cos^2\ell x.
\end{align*}
An asymptotic expansion shows that 
$u_1v_1 + u_2v_2|_{\partial \SC_{\g}} = O(1)$; this is 
actually quite remarkable, as both $u_1v_1$ and $u_2v_2$
are separately comparable to $\ell^{-3}$ on 
$\partial \SC_{\g}$. Moreover, 
$u_1v_1 + u_2v_2|_{x=0} = O(e^{-\sqrt{8}\pi/{\ell}}\ell^{-2})$.
With these computations in mind, we observe that the general 
solution to \eqref{BVP} is given by 
\begin{equation*}
u= c_1u_1 + c_2u_2+ u_1v_1 + u_2v_2.
\end{equation*}
By our estimates on $ u_1v_1 + u_2v_2|_{\partial \SC_{\g}}$ 
and the definitions of $u_1$ and $u_2$, we see that we may
take $c_1=O(\ell)$ and $c_2=O(\ell^2)$.  Thus we compute that
\begin{align*}
u(0) &= c_1u_1(0) + c_2u_2(0)+ u_1(0)v_1(0) + u_2(0)v_2(0)\\
&= c_2u_2(0)+ u_2(0)v_2(0) \qquad\text{since $u_1(0) =0$}\\
&=O(\ell) + O(e^{-\sqrt{8}\pi/{\ell}}\ell^{-3})\\
&=O(\ell).
\end{align*}
We conclude that 
$-2(\triangle-2)^{-1}\f{|\Phi|^2}{g^2}\Big|_{\g} =u_0\Big|_{\g} \le u(0) = O(\ell)$ and so
\begin{align*}
\int_{\g} -2(\triangle-2)^{-1}\f{|\Phi|^2}{g^2} = \int_{\g} u ds &\le
O(\ell)\int_{\g}ds\\
&=O(\ell^2).
\end{align*}

Combining the estimate for this term of the Hessian with the 
estimate for the other term of the Hessian discussed in 
Step 2 concludes the proof of the lemma.

\end{proof}

\section{The Thurston metric and the Weil-Petersson metric} \label{Thurston}

From the
formula \eqref{formula:d2L}, we can easily derive the result
\cite{wolpert:thurston} (see also \cite{mcmullen:thermodynamics} and 
\cite{bonahon:currents}) that the Thurston metric
is a multiple of the 
Weil-Petersson metric.

\subsection{The Thurston Metric} We begin by recalling the Thurston
metric. To define this, imagine a sequence $\{\g_n\}$ of closed curves
which are becoming equidistributed in the sense that if $B$ is a ball in the
unit tangent bundle $T^1S$, 
and we lift $\g_n$ to its representative in $T^1S$, then
$$
\lim_{n\to\infty}\f{\ell(\g_n\cap B)}{\ell(\g_n)} = 
\f{\text{Volume}(B)}{\text{Volume}(T^1S)}.
$$

Thurston noted that since, for such a sequence of curves,
$\f{d\ell(\g_n)}{\ell(\g_n)}\to 0$ as $n\to\infty$, then
$\f{\text{Hess}\ell(\g_n)}{\ell(\g_n)}$ would tend to a symmetric
quadratic form on $T(S)$; by the convexity of the length function,
this tensor would be positive semi-definite, hence a (pseudo)-metric. 
Wolpert showed that

\begin{thm} \cite{wolpert:thurston}. The Thurston metric is a 
multiple of the Weil-Petersson metric.
\end{thm}  \label{thurston wp}

The goal of the present section is to give a proof of this result that
proceeds from evaluating formula~\eqref{formula:d2L} on a 
sequence $\{\g_n\}$ of curves that are becoming equidistributed in $T^1S$.

\subsection{First Variation} We begin by first showing that
$d\ell(\g_n)/\ell(\g_n)\to0$ as $n\to\infty$. Recall from
\eqref{FirstVariation} that 
\begin{equation*}
\f d{dt}L(g_t,\g_{n,t})  = \int_{\g_n(0)}\f{\ree\Phi}{g_0}ds.
\end{equation*}

Now, the curves $\g_n$ have unit tangent vectors $\f d{ds}\g_n(s)$
which 
equidistribute themselves in $T^1S$, and so
\begin{align*}
\f1{\ell(\g_n)}\f d{dt}L(g_t,\g_{n,t})  
&= -\f1{\ell(\g_n)}\int_{\g_n}\f{\ree\Phi}{g_0}ds\\
&= -\int_{\g_n}\f{\ree\Phi}{g_0}\f{ds}{\ell(\g_n)}\\
&\rightarrow\iiint\limits_{T^1S}\f{\ree\Phi(p,\th)}{g_0}\f{d\vol_{T^1S}(p,\th)}{\vol(T^1S)}
\end{align*}
where we interpret the meaning of the notation as follows. In the
discussion so far, we have written $\ree\Phi$ to denote the value of
the expression $\ree\vp$, when the \qd $\Phi=\vp dz^2$ was written in
coordinates $z=x+iy$ with $\f\p{\p y}$ being tangent to the
geodesic. Now the vector field $\f\p{\p y}$ lifts to the canonical
vector fields in $T^1S$ tangent to the geodesic flow. We let the
expression $\ree\Phi(p,\th)$ denote the value of 
$\ree\Phi$ on the surface in terms of a coordinate $z=x+iy$ in 
which the geodesic direction described by $(p,\th) \in T^1S$
is in the coordinate direction $\f{\p}{\p y}$. 
Of course, if we
change surface coordinates so that $z_{\th}=e^{-i\th}z$, then for $\Phi=\vp_\th
dz^2_\th$, we have $\vp_\th=e^{2i\th}\vp_0$. 
Thus, when we integrate 
along the fiber of $T^1S \to S$
(with respect to $\th$ in the coordinates $(z,\th)$ for $T^1S$), we find 
$\int\f{\ree\vp(p,\th)}{g_0}d\th=0$.

\subsection{The second variation} We recall the (second) formula
\eqref{formula:d2L} for the second variation of length:
\begin{align*}
&\f1{\ell(\g_n)}\f{d^2}{dt^2}\ell(\g_n(t)) \\
&\hskip1cm= \f1{\ell(\g_n)}\int -(\Delta-2)^{-1}\f{2|\Phi|^2}{g^2_0}ds\\
&\hskip1.5cm+\f1{\ell(\g_n)}\iint_{\g_n \times \g_n}\im\mu(p)\f{\cosh(d(p,q) - \f{\ell(\g_n)}{2})}{2\sinh(\f{\ell(\g_n)}{2})}\im\mu(q)ds(p)ds(q)\\
&\hskip1cm= \f1{\ell(\g_n)}\int -2(\Delta-2)^{-1}\f{|\Phi|^2}{g^2_0}ds\\ 
&\hskip1.5cm+ \f1{\ell(\g_n)}\int_{\g_n}\f{\im\Phi(p)}{g_0(p)}\int_{\g_n}\f{\cosh(d(p,q) -
 \ell(\g_n)/2)}{2\sinh(\ell(\g_n)/2)}\f{\im\Phi(q)}{g_0(q)}dqdp\\
&\hskip1cm= I_1 + I_2.
\end{align*}

We examine the two integrals $I_1$ and $I_2$ separately. The first 
integral is immediate, as
$$
I_1 = \int_{\g_n} 
-2(\Delta-2)^{-1}\f{|\Phi|^2}{g^2_0}\f{ds}{\ell(\g_n)}\lra2\pi\iint\limits_M 
-2(\Delta-2)^{-1}\f{|\Phi|^2}{g^2_0}\f{d\area}{2\pi\area}
$$
by the equidistribution property, and
$$
\iint\limits_M -2(\Delta-2)^{-1}\f{|\Phi|^2}{g^2_0}\f{d\area}{\area} 
= \iint\{-2(\Delta-2)^{-1}(1)\}\f{|\Phi|^2}{g^2_0}\f{d\area}{\area}
$$
as the operator $(\Delta-2)^{-1}$ is self adjoint. Then 
as $-2(\Delta-2)^{-1}(1)=1$, we find that
\begin{equation}
I_1\lra\iint\f{|\Phi|^2}{g^2_0}\f{d\area}{\area}.\label{I1WP}
\end{equation} 

\subsection{The Integral $I_2$} We turn next to $I_2$, where the
computation is a bit more involved. Our basic plan mirrors our
discussion of the first variation; we extend the terms in the
integrand of $I_2$ to all of $T^1S$, and then integrate over the
circular fiber to be left with an integral 
over the surface.

We begin our more detailed discussion of the second (energy) integral
by considering the version \eqref{formula:KernelFormula} of it in terms of 
a geometric kernel, i.e.
$$
I_2 = \f1{\ell(\g_n)}\int_{\g_0}\f{\im\Phi(p)}{g_0(p)}\int_{\g_0}\f{\cosh(d(p,q) -
  L/2)}{2\sinh(L/2)}\f{\im\Phi(q)}{g_0(q)}dqdp
$$
where $L=\ell(\g_0)$. 

Now considering $\g_0$ as an embedded curve in the unit tangent bundle
$T^1(S,g_0)$, if we fix the point $p=(\bar p, v) \in T^1S$ as
representing a point $\bar p \in S$ and a unit vector $v\in T_p^1S$,
then a point $q = (\bar q, w)$ along $\g_0$ at distance $t$ from
$p$ could be written 
$$
q=\exp_{\bar p}tv = G_tp
$$
where $G_t$ denotes the geodesic flow in $T^1S$ for distance $t$.

Thus we see that as $L \to \infty$, the integral $I_2$ converges to 
$$
\lim_{L \to \infty} I_2 =
\int_{T^1S}\f{\im\Phi(p)}{g_0(p)}\int_0^{\infty}\f{e^{-t}}{2}(\f{\im\Phi(G_tp)}{g_0(G_t(p))}
+ \f{\im\Phi(G_{-t}p)}{g_0(G_{-t}(p))})\f{dtdp}{\vol(T^1S)}.
$$

Of course, as $p$ varies in the fiber $\{(\bar p, e^{i\theta}v)\}$
over $\bar p \in S$, we observe the points $G_t(\bar p,v)$ also
arising as $G_{-t}(\bar p,-v)$, and so we may rewrite the above limit
integral as
$$
\lim_{L \to \infty} I_2 =
\int_{T^1S}\f{\im\Phi(p)}{g_0(p)}\int_0^{\infty}e^{-t}\f{\im\Phi(G_tp)}{g_0(G_t(p))}
\f{dtdp}{\vol(T^1S)}. 
$$

To evaluate this last integral, imagine $\bar p =0$ in the disk 
$\{|z|<1\}$ and we represent the hyperbolic metric as 
$$
g_0 = \f{4|dz|^2}{(1-r^2)^2} = g_0(0)(1-r^2)^{-2}|dz|^2.
$$
An important matter here (as it was in the calculation of the 
first variation) is the question of how to interpret
the meaning of $\im\Phi(q)$ in these coordinates: recall 
that we understood $\im\f{\Phi}{g_0}$ to be the value
of $\im\f{\Phi}{g_0}$ when we defined the geodesic 
$\g_0$ as a vertical line in the coordinate system.
In the fixed coordinate system of the disk $\{|z|<1\}$,
write $\Phi = \phi(z)dz^2$; then since the hyperbolic geodesic
through the origin and the point $z=re^{i\theta}$ is given
by the ray $te^{i\theta}$, we see that 
$\im\Phi|_{re^{i\theta}}= \im e^{2i\theta}\phi(re^{i\theta})$.
With this notation, our integral becomes

\begin{align*}
&\lim_{L \to \infty} I_2 \\
&= \int_S\int_0^{2\pi}\f{\im
  e^{2i\theta}\phi(0)}{g_0(p)}e^{-t}\f{\im
  e^{2i\theta}\phi(r(t)e^{i\theta})}{g_0(p)(1-r(t)^2)^{-2}}\f{dtd\theta d\area}{2\pi\area(S)}\\
&=\int_S\int_0^{\infty}\f{e^{-t}(1-r(t)^2)^2}{g_0^2(p)}\int_0^{2\pi}\im(e^{2i\theta}\phi(0))\im(e^{2i\theta}\phi(r(t)e^{i\theta}))\f{dtd\theta
    d\area}{2\pi\area(S)}.
\end{align*}

Using the change-of-coordinates 
formula $e^{-t}dt = 2(1+r)^{-2}dr$ and the mean value theorem for
harmonic functions (together with a half-angle formula to simplify the
averaging), we find

\begin{align}\label{I2WP}
\lim_{L \to \infty} I_2 &= \int_S\f{1}{g_0(p)^2}\biggl(\int_0^1
2(1-r)^2dr\biggr)\f{|\phi(0)|^2}{2}\f{d\area}{\area(S)} \notag \\
&= \f{1}{3}\int_S \f{|\Phi(p)|^2}{g_0(p)^2}\f{d\area(p)}{\area(S)}.
\end{align}

Combining \eqref{I1WP} and \eqref{I2WP}, we verify
Theorem~\ref{thurston wp}.
In particular,

\begin{align*}
\lim_{n \lra \infty}\f{\text{Hess}\ell(\g_n)}{\ell(\g_n)} &=
\f{4}{3}\int_S \f{|\Phi(p)|^2}{g_0(p)^2}\f{d\area(p)}{\area(S)}\\
&=\f{4}{3\area(S)}\|\mu\|_{\text{WP}}^2.
\end{align*}

\begin{remark}
The constant $\f{4}{3}$ found here agrees with that found by Wolpert
\cite{wolpert:thurston}
and McMullen \cite{mcmullen:thermodynamics}.  See the comments (\cite{mcmullen:thermodynamics}, p.376) of McMullen
on the consistency of conventions.
\end{remark}

\bibliographystyle{alpha} 
\bibliography{liter12-08}

\end{document}